\documentclass[11pt,english]{amsart}
\usepackage[utf8]{inputenc}
\usepackage[english]{babel}

\usepackage[letterpaper]{geometry}
\usepackage{verbatim}
\usepackage{amstext}
\usepackage{amsthm}
\usepackage{amssymb}
\usepackage{amsmath}
\usepackage{stmaryrd}
\usepackage{xcolor}
\usepackage{hyperref}
\usepackage[all,cmtip]{xy}
\usepackage{bm}
\usepackage{etoolbox}
\usepackage{dynkin-diagrams}

\usetikzlibrary{decorations.pathreplacing}

\newtheorem{thm}{\protect\theoremname}[section]
\newtheorem*{thm*}{\protect\theoremname}
  \theoremstyle{plain}
  \newtheorem{prop}[thm]{\protect\propositionname}
  \theoremstyle{definition}
  \newtheorem{example}[thm]{\protect\examplename}
  \theoremstyle{plain}
  \newtheorem{cor}[thm]{\protect\corollaryname}
  \theoremstyle{definition}
  \newtheorem{defn}[thm]{\protect\definitionname}
  \theoremstyle{remark}
  \newtheorem{rem}[thm]{\protect\remarkname}
  \theoremstyle{plain}
  \newtheorem{lem}[thm]{\protect\lemmaname}

\makeatletter
\@namedef{subjclassname@2020}{%
  \textup{2020} Mathematics Subject Classification}
\makeatother

\title{Root data in character varieties}

\subjclass[2020]{Primary: 14M35. Secondary: 14L24, 14D20, 14C15, 17B22}
\keywords{Character varieties, Root data, Geometric Invariant Theory, Langlands duality.}

\author[\'A. Gonz\'alez-Prieto]{\'Angel Gonz\'alez-Prieto}
 \address{Departamento de
  \'Algebra, Geometr\'ia y Topolog\'ia, Facultad de Ciencias Matemáticas, Universidad Complutense de Madrid, Plaza Ciencias 3, 28040 Madrid, Spain}
  \address{Instituto de Ciencias Matem\'aticas (CSIC-UAM-UCM-UC3M), C/ Nicol\'as Cabrera 13-15, 28049 Madrid, Spain}
  \email{angelgonzalezprieto@ucm.es}
  
\author[A. Zamora]{Alfonso Zamora}
 \address{Departamento de Matemática Aplicada a las TIC, ETSI Informáticos, Universidad Politécnica de Madrid, Campus de Montegancedo, 28660 Madrid, Spain}
  \email{alfonso.zamora@upm.es}

\renewcommand{\hom}{\mathrm{Hom}}
\def\quot{/\!\!/}
\def\sym{\mathsf{Sym}}

\newcommand{\PP}{\mathbb{P}}
\newcommand{\ZZ}{\mathbb{Z}}

\newcommand{\CC}{\mathbb{C}}

\newcommand{\GG}{\mathbb{G}}

\newcommand{\FF}{\mathbb{F}}

\newcommand{\Spec}{\operatorname{Spec}}

\newcommand{\Hom}{\operatorname{Hom}}

\newcommand{\GL}{\operatorname{GL}}
\newcommand{\SL}{\operatorname{SL}}
\newcommand{\SO}{\operatorname{SO}}
\newcommand{\PGL}{\operatorname{PGL}}

\newcommand{\Sp}{\operatorname{Sp}}

\newcommand{\KVar}{\textup{K}\textbf{\textup{Var}}}

\newcommand{\Free}[1]{\textbf{\textup{F}}_{{#1}}}

\providecommand{\corollaryname}{Corollary}
  \providecommand{\definitionname}{Definition}
  \providecommand{\examplename}{Example}
  \providecommand{\lemmaname}{Lemma}
  
  \providecommand{\propositionname}{Proposition}
  \providecommand{\remarkname}{Remark}
\providecommand{\theoremname}{Theorem}

\makeatother

\begin{document}

\begin{abstract}
Given $G$ an algebraic reductive group over an algebraically closed field of characteristic zero and $\Gamma$ a finitely generated group, we provide a stratification of the $G$-character variety of $\Gamma$ in terms of conjugacy classes of parabolic subgroups of $G$. Each stratum has the structure of a pseudo-quotient, which is a relaxed GIT notion capturing the topology of the quotient and, therefore, behaving well for motivic computations of invariants of the character varieties. These stratifications are constructed by analyzing the root datum of $G$ to encode parabolic classes. Finally, detailed and explicit motivic formulae are provided for cases with Dynkin diagrams of types $A$, $B$, $C$ and $D$. 
\end{abstract}

\maketitle

\begin{center}
    \textit{Dedicated to the memory of Prof.\ Peter E.\ Newstead.}
\end{center}

\section{Introduction}

Let $G$ be an algebraic reductive group and let $\Gamma = \langle \gamma_1, \ldots, \gamma_s\,\mid r_{\alpha}(\gamma_1, \ldots, \gamma_s) = 1\rangle$ be a finitely generated group, where $\gamma_1, \ldots, \gamma_s$ are the generators and $r_{\alpha}$ are the relations satisfied by these generators.
Denote by $\mathcal{R}_G(\Gamma)=\hom(\Gamma,G)$ the set of group homomorphisms $\rho: \Gamma \to G$. This set has the natural structure of an affine algebraic variety given by the natural identification with the algebraic set
$$
    \mathcal{R}_G(\Gamma) \cong \left\{(g_1, \ldots, g_s) \in G^s\,\mid\, r_{\alpha}(g_1, \ldots, g_s) = 1\right\}.
$$
With the algebraic structure described above, the variety
$$
    \mathcal{R}_G(\Gamma) = \hom(\Gamma,G)
$$
is known as the \emph{$G$-representation variety} of $\Gamma$.     
The geometry of these representation varieties has been deeply studied in the literature, especially when $\Gamma = \pi_1(M)$ is the fundamental group of a certain manifold $M$, as in the case of surfaces \cite{hausel2008mixed,logares2013hodge,mereb2015polynomials,martinez2016polynomials,gonzalez2020virtual} or $3$-dimensional manifolds \cite{munoz2016geometry,gonzalez2022motive,gonzalez2023representation}. It has also been studied for nilpotent groups \cite{florentino2024mixed}, among others.

Despite its importance, the representation variety only parametrizes the whole set of representations, without identifying isomorphic ones. To address this issue, we consider the action of the group $G$ on $\mathcal{R}_G(\Gamma)$
by conjugation of representations, i.e.\ $(g \cdot \rho)(\gamma) = g \rho(\gamma)g^{-1}$ for $g \in G$, $\rho \in \mathcal{R}_G(\Gamma)$ and $\gamma \in \Gamma$. 
In this way we can define the Geometric Invariant Theory (GIT) quotient    
\[
\mathcal{X}_G(\Gamma) =\mathcal{R}_G(\Gamma)\quot G,
\]
which is an algebraic variety known as the \emph{$G$-character variety} of $\Gamma$.

Since $\mathcal{X}_G(\Gamma)$ parametrizes isomorphism classes of representations, this space is also known as the moduli space of representations of $\Gamma$
into $G$ or the Betti moduli space in the context of non-abelian Hodge correspondence \cite{corlette1988flat,simpson1992higgs,simpson1994moduliI,simpson1994moduliII}, in particular in the celebrated $P=W$ conjecture \cite{hausel2022p,maulik2022p}. These character varieties play a central role in modern geometry, as the wide literature about them shows, such as \cite{hausel2011arithmetic,de2012topology,baraglia2017arithmetic,gonzalez2022motive,mellit2020poincare} for surfaces, \cite{culler1983varieties,gonzalez2024character} for $3$-manifolds, or \cite{lawton2008minimal,lawton2009poisson,florentino2009topology,florentino2012singularities,florentino2014topology,lawton2016polynomial,casimiro2016topology,florentino2021serre,florentino2021hodge} for free (possibly abelian) groups, among many other works. 

In the case when $G = \GL_n$, the homomorphisms $\rho: \Gamma \to \GL_n$ are genuine linear representations of rank $n$, which allows us to use some linear algebra tools to understand the character variety, such as reducibility and semisimplicity. Recall that a representation $\rho$ is said to be reducible if there exists a proper subspace $0 \neq V \subsetneq \CC^n$ such that $\rho(\gamma)(V) \subseteq V$ for all $\gamma \in \Gamma$; otherwise we say that $\rho$ is irreducible. Irreducible representations are very well-behaved, for instance, their isotropy group for the conjugacy action is the smallest possible one, and they are stable points for the GIT quotient.

Furthermore, in this case $G=\GL_{n}$, the whole character variety can be understood in terms of semisimple representations, i.e.\ direct sums of irreducible representations. This natural stratification of $\mathcal{X}_{\GL_n}(\Gamma)$ was constructed in \cite{FNZ2023generating} and is indexed by what the authors call the `partition type'.
Let us briefly review it to understand the flavour of the results to come. A partition of $n\in\mathbb{N}$ is
denoted by $[k]=[1^{k_{1}}\cdots j^{k_{j}}\cdots n^{k_{n}}]$, where
the exponent $k_{j}$ means that $[k]$ has $k_{j}\geq0$ parts of
size $j\in\{1,\ldots,n\}$, so that $n=\sum_{j=1}^{n}j\cdot k_{j}$. For example, $[1^{3}\,2\, 4]$ is the partition $9=1\cdot 3+2+4$.
The set of partitions of $n$ will be denoted by $\mathcal{P}_{n}$. In this way, we say that $\rho\in\mathcal{R}_{\GL_n}(\Gamma)$
is {$[k]$-polystable} if $\rho$ is conjugate to a representation of the form
\begin{equation*}
\bigoplus_{j=1}^{n}\rho_{j}\label{eq:polystable-type}
\end{equation*}
where each $\rho_{j}$ is a direct sum of $k_{j} \geq 0$ irreducible
representations of $\mathcal{R}_{\GL_{j}}(\Gamma)$, for $j=1,\ldots,n$. We denote $[k]$-polystable representations by
$\mathcal{R}^{[k]}_{\GL_n}(\Gamma) \subseteq \mathcal{R}_{\GL_n}(\Gamma)$ and analogously
for their equivalence classes under conjugation by $\mathcal{X}^{[k]}_{\GL_n}(\Gamma)\subseteq \mathcal{X}_{\GL_n}(\Gamma)$.

It can be proven that each stratum $\mathcal{R}^{[k]}_{\GL_n}(\Gamma)$ is polystable (i.e.\ their orbits in the GIT quotient are closed) and actually the union of these strata is the polystable locus of $\mathcal{R}_{\GL_n}(\Gamma)$. Furthermore, the stable GIT locus of $\mathcal{R}_{\GL_n}(\Gamma)$ is precisely the open stratum $\mathcal{R}^{[n]}_{\GL_n}(\Gamma) = \mathcal{R}^{\star}_{\GL_n}(\Gamma)$ of irreducible representations, corresponding to the trivial partition $[n] = [n^{1}]$. With this information, in \cite[Proposition 4.3]{FNZ2023generating} the authors proved that the character variety $\mathcal{X}_{\GL_n}(\Gamma)$ can be written as a
disjoint union, labelled by partitions $[k]\in\mathcal{P}_{n}$, of
locally closed quasi-projective varieties of $[k]$-polystable equivalence
classes
\begin{equation}\label{eq:decomposition-gl}
\mathcal{X}_{\GL_n}(\Gamma)=\bigsqcup_{[k]\in\mathcal{P}_{n}}\mathcal{X}^{[k]}_{\GL_n}(\Gamma).
\end{equation}

This decomposition can be better understood through Levi subgroups. For each partition $[k]\in\mathcal{P}_{n}$, denote by $L_{[k]}$ the reductive subgroup 
\begin{equation*}
L_{[k]} =\GL_{1}^{k_{1}}\times\cdots\times \GL_{n}^{k_{n}}\subseteq \GL_{n},\label{eq:k-Levi}
\end{equation*}
which we call the $[k]$-Levi of $\GL_{n}$. In fact, all Levi subgroups
of $\GL_{n}$ are conjugate to one obtained in this way. In this way, the $[k]$-polystable representations $\mathcal{R}^{[k]}_{\GL_n}(\Gamma)$ are exactly the representations $\rho: \Gamma \to L_{[k]}$ that are irreducible as $L_{[k]}$-representations. Furthermore, if we consider the group $N_{[k]} = L_{[k]} \rtimes S_{[k]}$, where $S_{[k]} = S_{k_1} \times \ldots \times S_{k_n} \subseteq S_n$ acts on $L_{[k]}$ by permutation of blocks of equal size, we have that $N_{[k]}$ is exactly the $\GL_n$-normalizer of $L_{[k]}$. In this manner, we can actually identify
$$
\mathcal{X}^{[k]}_{\GL_n}(\Gamma) \cong \mathcal{R}^{[k]}_{\GL_n}(\Gamma)\quot (L_{[k]}\rtimes S_{[k]}) = \left(\prod_{j=1}^{n}\mathcal{X}^{\star}_{\GL_j}(\Gamma)^{k_{j}}\right)/S_{[k]},
$$
where $S_{[k]} = N_{[k]} / L_{[k]}$ can be thought of as some sort of Weyl group associated to the partition $[k]$.

The goal of this work is to generalize the decomposition (\ref{eq:decomposition-gl}) to a general reductive group $G$ over an algebraically closed field of characteristic zero. To this aim, we shall need to rephrase many of the statements above in an intrinsic way not depending on the linear embedding of $G$. In this direction, it turns out that the theory of parabolic and Levi subgroups of $G$ provides the perfect framework to set these ideas.

Once we fix a Borel subgroup $B \subseteq G$ and a maximal torus $T \subseteq G$, this defines a finite set of simple roots $\Delta$ of $G$, and subsets $I \subseteq \Delta$ are in correspondence with standard parabolic subgroups $P_I \subseteq G$, i.e.\ parabolic subgroups containing $B$, as well as subtori $T_I \subseteq T$. The Levi subgroup $L_I$ associated to $P_I$ is thus the centralizer of $T_I$, and is a reductive group whose root system is generated by the roots of $I$. Attached to this subset we also find the normalizer $N_I$ of $T_I$ and the Weyl group of $I$, $W_I := N_I / L_I$.

In this spirit, we will say that a representation $\rho: \Gamma \to G$ of $G$ is {reducible} if there exists a proper parabolic subgroup $P \subsetneq G$ such that $\rho(\Gamma) \subseteq P$; and irreducible otherwise. The set of irreducible representations forms an open set $\mathcal{R}_G^{\star}(\Gamma) \subseteq \mathcal{R}_G(\Gamma)$. Taking into account that parabolic subgroups of $\GL_n$ are in correspondence with stabilizers of flags of $\CC^n$ (see Section \ref{sec:parabolic-stratification} for details) this definition of irreducibility can be seen as a generalization of the classical notion for a general reductive group $G$. In this form, we can consider $\mathcal{R}_{L_I}^\star(\Gamma) \subseteq \mathcal{R}_G(\Gamma)$  to be the collection of representations $\rho$ with $\rho(\Gamma) \subseteq L_I$ and such that $\rho$ is irreducible as an $L_I$-representation. 

With these notions, the main result of this paper is the following decomposition. Here, we will work on the Grothendieck ring $\KVar$ of algebraic varieties, generated by isomorphism classes $[X]$ of algebraic varieties, called virtual classes, modulo cut-and-paste relations.

\begin{thm*}[Corollary \ref{cor:main-result-good-situation}]
Let $G$ be a reductive group and let $\Gamma$ be any finitely generated group. Suppose that, for any subset $I \subseteq \Delta$ of simple roots we can decompose the normalizer as a semidirect product $N_I = L_I \rtimes W_I$ of the Levi and the Weyl groups of $I$. Then, we have that the virtual class in $\KVar$ of the character variety $\mathcal{X}_G(\Gamma)$ can be written as
\begin{equation}\label{eq:decomposition-intro}
    \left[\mathcal{X}_G(\Gamma)\right] = \sum_{I \in 2^{\Delta}/\sim_W}  \left[\mathcal{X}_{L_I}^{\star}(\Gamma) \sslash W_I\right],
\end{equation}
where the sum runs over equivalence classes of subsets $I \subseteq \Delta$ where $I \sim_W I'$ if there is an element of $W$ taking the root subsystem generated by $I$ to that of $I'$.
\end{thm*}

Notice that $W_I$ is a finite group so in particular all the quotients appearing in the previous formula are regular quotients. It is worth mentioning that the hypothesis that we can decompose $N_I = L_I \rtimes W_I$ is not very restrictive, and indeed it is fulfilled by most classical groups. Furthermore, in the case this condition does not hold for a certain $I$, we can still get a decomposition by replacing the corresponding summand by $\mathcal{R}_{L_I}^\star(\Gamma) \sslash N_I$.

The proof of this result will make crucial use of the theory of pseudo-quotients for GIT, as developed in \cite{gonzalez2024pseudo}. In particular, a key result will be to show that $(\mathcal{R}_{L_I}^{\star}(\Gamma), N_I)$ is a core for the action (cf.\ Section \ref{ssec:cores}). Roughly speaking, this means that any orbit of a $P_I$-representation has an element of $\mathcal{R}_{L_I}^\star(\Gamma)$ in its closure, and two representations of $\mathcal{R}_{L_I}^{\star}(\Gamma)$ are $G$-equivalent if and only if they are $N_I$-equivalent. To address this problem, we will need several subtle results about Levi and parabolic subgroups of a reductive group.

As an application of the main result, we shall obtain novel decompositions for groups of Dynkin diagram of type $A$, $B$, $C$ and $D$. In particular, we study in full detail the cases where $G$ is $\GL_n, \SL_n$ and $\PGL_n$ (Dynkin diagram $A_{n-1}$), $\SO_{2n+1}$ (Dynkin diagram $B_n$), $\Sp_{2n}$ (Dynkin diagram $C_n$) and $\SO_{2n}$ (Dynkin diagram $D_n$), but a similar analysis can be applied to obtain decompositions for more general groups.

Apart from providing a better understanding of the structure of character varieties, this work also aims to address a fundamental question: What is the role of the root datum of $G$ in the $G$-character variety? So far, the available techniques to study character varieties exploit no information of the root data. The arithmetic techniques, as developed by Hausel and Rodr\'iguez-Villegas in \cite{hausel2011arithmetic}, only make use of the representation theory of the finite group $G(\FF_q)$ over the finite field of $q$ elements; whereas the geometric techniques, developed by Logares, Muñoz and Newstead in \cite{logares2013hodge}, take advantage only of the geometry of the orbit space $G/G$. Despite this information being obviously linked with the root datum of $G$, the connection has never been cleared up. 

Revealing the involvement of the root datum in the character variety is definitely a first step towards a better understanding of the geometric Langlands program and the mirror symmetry conjectures \cite{beilinson1991quantization,hausel2003mirror,kapustin2006electric}. Recall that given an algebraic group $G$, its Langlands dual ${^L}G$ is the only reductive group whose root datum is the dual root datum of $G$. In this way, the geometric mirror symmetry conjecture, as proposed in \cite{hausel2003mirror}, predicts an equality of virtual classes in $\KVar$
$$
    \left[\mathcal{X}_G(\Gamma)\right] = \left[\mathcal{X}_{{^L}G}(\Gamma)\right].
$$
Several weaker versions can also be considered, such as restricting ourselves to the case where $\Gamma$ is the fundamental group of a surface, considering a coarser cohomological invariant than virtual classes known as the $E$-polynomial, or twisting this equality by considering a modified version known as the stringy $E$-polynomial. To the best of our knowledge, this conjecture has only been verified in some particular cases, see \cite{martinez2017polynomial,groechenig2020mirror,maulik2021endoscopic,mauri2021topological,florentino2021serre}.

In this direction, the decompositions obtained in this work provide strong evidence supporting the geometric mirror symmetry conjecture. Indeed, the sum in the stratification (\ref{eq:decomposition-intro}) is indexed by the same objects for Langlands dual groups, so the conjecture can be reduced to verifying that
$$
    \left[\mathcal{X}_{L_I}^\star(\Gamma) \sslash W_I\right] = \left[\mathcal{X}_{{^L}L_I}^\star(\Gamma) \sslash W_I\right]
$$
for any subset $I \subseteq \Delta$ of simple roots. Furthermore, the Levi subgroups $L_I$ and ${^L}L_I$ tend to be much simpler than the original group $G$, since their Dynkin diagram is a subset of the original one. In particular, in many cases, the resulting Levi subgroups are of type $A$ and the predicted symmetry reduces to the symmetry for $G = \SL_n$ and ${^L}G = \PGL_n$. In any other cases, the Levi subgroup may have a Dynkin diagram of the same shape as the original group, but with fewer vertices, so the symmetry easily follows from our results by induction on the rank of the group.

The only stratum that remains elusive to this argument is precisely the loci $\mathcal{X}_{G}^\star(\Gamma) $ and $\mathcal{X}_{{^L}G}^\star(\Gamma)$ of irreducible representations, that correspond to the largest stratum $I = \Delta$.
In this way, as by-product of this argument and the decompositions provided for the $A,B,C$ and $D$ cases, we also obtain that the geometric mirror symmetry conjecture holds for an arbitrary $G$ if and only if it holds for $G$ of type $A$ and for the irreducible locus of the $G$-character variety. Furthermore, with the techniques developed in this paper, we expect to address the exceptional group cases as well as other Langlands dual pairs in a future work.

\subsection*{Structure of the manuscript} In Section \ref{sec:GIT_and_pseudo} we review the main results of Geometric Invariant Theory and pseudo-quotients that will be used throughout the paper. Section \ref{sec:reductive-groups-general} is devoted to the study of root data in reductive groups, including some auxiliary results that can be of independent interest, with special attention to root data of type $A$, $B$, $C$ and $D$. Section \ref{sec:parabolic-stratification} is the core of this paper, and there we prove the main results of this work in Theorem \ref{thm:main-result} and Corollary \ref{cor:main-result-good-situation}. Finally, in Section \ref{sec:parabolic-stratification-classical-groups} we apply the previous results to explicitly state these decompositions for some families of classical groups.

\subsection*{Acknowledgements}

The authors acknowledge Will Sawin for useful conversations regarding the realization of parabolic subgroups as flags. The key idea for the proof of Proposition \ref{prop:normalizer-conjugation} was kindly provided to us by Sean Cotner (SeanC user) in the MathOverflow forum \cite{SeanCMO}. We also thank the referee for their thorough review and very constructive suggestions and references that have certainly improved the manuscript, including the proof provided in Remark \ref{rmk:parabolic-of-parabolic}.

The first-named author has been partially supported by Spanish Ministerio de Ciencia e Innovaci\'on project PID2021-124440NB-I00, through the COMPLEXFLUIDS grant awarded by the BBVA Foundation, and by the Madrid Government (\textit{Comunidad de Madrid} – Spain) under the Multiannual
Agreement with the Universidad Complutense de Madrid in the line Research Incentive for Young PhDs, in the context of the V PRICIT (Regional Programme of Research and Technological Innovation) through the project PR27/21-029.
The second-named author has been partially supported by Spanish Ministerio de Ciencia, Innovaci\'on e Universidades project PID2022-142024NB-I00.

\section{Geometric Invariant Theory and pseudo-quotients}
\label{sec:GIT_and_pseudo}

In this section, we shall briefly review some of the fundamental results of Geometric Invariant Theory (GIT) that will be used in this work. For a complete introduction on the techniques to be used, see \cite{newstead1978introduction} and \cite{gonzalez2024pseudo}. For simplicity, throughout this paper we shall work on an algebraically closed field of characteristic zero, mainly over the complex numbers $\CC$, but many results can be adapted to work in positive characteristic, see Remark \ref{rem:positive-char}.

\subsection{Review of GIT}
\label{ssec:reviewGIT}

Let $X$ be an algebraic variety and let $G$ be an algebraic group acting on $X$. Denote by $G\cdot x$ the orbit of a point $x\in X$ by $G$ and by $\mathcal{O}_{X}$ the sheaf of regular functions on $X$. 

\begin{defn}
A pair $(Y,\pi)$, where $Y$ is an algebraic variety and $\pi:X \to Y$ is a regular $G$-invariant morphism, is called a \emph{categorical quotient} for the action of $G$ on $X$ if, for any other regular $G$-invariant morphism $f:X\to Z$ into an algebraic variety $Z$, there exists a unique morphism $g:Y\to Z$ such that the following diagram commutes
$$
\xymatrix{
X \ar[r]^\pi\ar[rd]_f & Y \ar@{--{>}}[d]^g\\
& Z}
$$
\end{defn}

\begin{defn}
The pair $(Y,\pi)$ where $\pi:X \to Y$ is a regular morphism is called a \emph{good quotient} of $(X,G)$ if:
\begin{enumerate}
\item $\pi$ is surjective.
\item $\pi$ is $G$-invariant.
\item If $W\subseteq X$ is a closed $G$-invariant subset, then the image $\pi(W)\subseteq Y$ is closed. 
\item For every two closed $G$-invariant subsets $W_1, W_2\subseteq X$, we have that $W_1\cap W_2=\emptyset $ if and only if $\pi(W_1)\cap \pi(W_2)=\emptyset$. 
\item For every open set $V\subseteq Y$, $\pi$ induces an isomorphism $\pi^{\ast}:\mathcal{O}_{Y}(V)\cong \mathcal{O}_{X}(\pi^{-1}(V))^{G}\subseteq \mathcal{O}_{X}(\pi^{-1}(V))$ with the subring of $G$-invariant functions. 
\end{enumerate}
\end{defn}

\begin{rem}
If $(Y, \pi)$ is a good quotient, then it is a categorical quotient (cf.\ \cite[Corollary 3.5.1]{newstead1978introduction}). Moreover, categorical quotients, when they exist, are unique.
\end{rem}

\begin{defn}
The pair $(Y,\pi)$ is called a \emph{geometric quotient} for the action of $G$ on $X$ if it is a good quotient which is an orbit space, i.e. $G\cdot x$ is closed in $X$ for every $x\in X$. 
\end{defn}

When the group $G$ is reductive, there exists a procedure to construct good quotients for its action on $X$, known as Geometric Invariant Theory (GIT). First, suppose that $X = \Spec(A)$ is affine. Then, the GIT quotient of $X$ by $G$ is
$$
    X \sslash G = \Spec(A^G),
$$
where $A^G$ denotes the subring of $G$-invariant functions on $A$, which is finitely generated by a theorem of Nagata \cite{nagata1963invariants}. The inclusion map $A^G \hookrightarrow A$ induces a morphism $X \to X \sslash G$ that can be proven to be a good quotient for the action (cf.\ \cite[Theorem 3.5]{newstead1978introduction}).

In the general case of a quasi-projective variety $X$, to carry out this construction, we need an extra piece of information given by a linearization of the action, which is essentially an embedding of $X$ into a projective space $\PP^N$ such that the action on $X$ is the restriction of a linear action of $G$ on $\PP^N$. A linearization defines an open $G$-invariant subset $X^{\textup{SS}} \subseteq X$, called the \emph{semistable} locus and, on it, the same construction can be performed by gluing together GIT quotients constructed on each affine patch, giving rise to a good quotient on $X^{\textup{SS}}$, also denoted by $X^{\textup{SS}} \sslash G$. Furthermore, there exists another open set $X^{\textup{S}} \subseteq X^{\textup{SS}}$, called the \emph{stable} locus on which the GIT quotient is a geometric quotient. Moreover, there exists a bigger subset $X^{\textup{PS}}$, called the \emph{polystable} locus, with $X^{\textup{S}} \subseteq  X^{\textup{PS}} \subseteq X^{\textup{SS}}$ but not necessarily open, on which the GIT quotient is also a good quotient. If $X$ is affine, then it is naturally endowed with a linearization and for this linearization we have $X^{\textup{SS}} = X$. For further details, please refer to \cite[Chapter 3]{newstead1978introduction}.

\subsection{Pseudo-quotients}
\label{ssec:pseudo_quotients}

Despite their importance and geometric interpretation, good quotients are not well behaved motivically. For instance, they may not commute with stratifications. In order to get a better-behaved quotient, we should use pseudo-quotients instead, as introduced in \cite{gonzalez2024pseudo}. These pseudo-quotients are a weaker notion of good quotient capturing the topological relations but omitting the algebraic correspondence between $G$-invariant functions. 

\begin{defn}
A pair $(Y,\pi)$ is called a \emph{pseudo-quotient} for the action of an algebraic group $G$ on $X$ if 
$\pi:X \to Y$ is a regular morphism satisfying properties (1)-(4) in the definition of good quotient. 
\end{defn}

\begin{rem}
For every open set $V\subseteq Y$, $\pi$ induces a morphism 
$\pi^{\ast}:\mathcal{O}_{Y}(V)\rightarrow \mathcal{O}_{X}(\pi^{-1}(V))$ which factors through $\mathcal{O}_{X}(\pi^{-1}(V))^{G}\subseteq \mathcal{O}_{X}(\pi^{-1}(V))$ because $\pi$ is $G$-invariant. However, in a pseudo-quotient, this map $\pi^{\ast}:\mathcal{O}_{Y}(V) \to \mathcal{O}_{X}(\pi^{-1}(V))^{G}$ does not have to be an isomorphism, as required in property (5) of a good quotient. 
\end{rem}

Pseudo-quotients might not be unique (in particular, they are not categorical quotients), since they do not capture the algebraic information of the structure sheaf of the quotient. However, it turns out that they are unique up to cut-and-paste relationships, as formalized by the Grothendieck ring of algebraic varieties.

\begin{defn}
The \emph{Grothendieck ring of algebraic varieties} $\KVar$ is the ring generated by isomorphism classes $[X]$ of algebraic varieties $X$, modulo the so-called cut-and-paste relations
$$
    [X] = [Y] + [X-Y],
$$
for any closed subvariety $Y\subseteq X$. Multiplication is given by cartesian product of varieties.
\end{defn}

Let us review some important features of pseudo-quotients that will be useful in the following.

\begin{prop}\cite[Corollary 3.7, Proposition 3.8, and Corollary 4.3]{gonzalez2024pseudo}
Let $(Y, \pi)$ be a pseudo-quotient of the action of $G$ on $X$. 
\begin{enumerate}
\item If $Y$ is normal and $X$ admits a pseudo-quotient $\xi: X\to Z$ which is also a categorical quotient, then $Y$ is isomorphic to $Z$ and $\pi$ is also a categorical quotient. 
\item If $X$ is irreducible and $Y$ is normal, then $(Y,\pi)$ is a good quotient.
\item Any other pseudo-quotient $(Z,\xi)$ verifies that $[Y]=[Z]$ in $\KVar$. 
\end{enumerate}
\end{prop}

\begin{rem}\label{rem:positive-char}
The previous proposition requires that we are working on an algebraically closed field of characteristic zero. However, result (3) remains valid in positive characteristic if we pass to the Grothendieck ring of constructible sets.
\end{rem}

\begin{rem}
In the case of complex algebraic varieties, a slightly coarser invariant can also be considered. Given a complex algebraic variety $X$, recall that its compactly-supported cohomology $H^k_c(X; \CC)$ is naturally equipped with a mixed Hodge structure, which provides the data of an increasing filtration $W_\bullet$, called the weight filtration, as well as a decreasing filtration $F^\bullet$, called the Hodge filtration. With this information at hand, we define the $E$-polynomial of $X$ as 
$$
    e(X) = \sum_{k,p,q}(-1)^k h_c^{k;p,q}(X) \,u^pv^q \in \ZZ[u,v],
$$
where $h_c^{k;p,q}(X) = \dim_\CC \textup{Gr}_p^{F^\bullet}\textup{Gr}^{p+q}_{W_\bullet}H_c^k(X;\CC)$ are the so-called Hodge numbers of $X$, with $\textup{Gr}^{\ell}_{W_\bullet}H_c^k(X;\CC) = W_\ell \left(H_c^k(X;\CC)\right) / W_{\ell-1} \left(H_c^k(X;\CC)\right)$ being the graded complex of $W_\bullet$, and similarly for $F^\bullet$. This invariant satisfies the cut-and-paste relations, so it descends to a ring homomorphism $e: \KVar \to \ZZ[u,v]$.
\end{rem}

\subsection{Cores}
\label{ssec:cores}

The idea of a core is the following. Given an algebraic variety $X$ and an action by $G$, suppose that there exists a subvariety $Y\subseteq X$ such that the closure of each $G$-orbit, $\overline{G\cdot x}$, intersects $Y$. This way, all the equivalence classes of points in a good quotient of $X$ by $G$ contain a representative in the subvariety $Y$. However, $Y$ does not need to be a slicing: it must be quotiented by the subgroup $H \subseteq G$ leaving $Y$ invariant. 

\begin{defn}\label{defn:pseudo-quotient}
Let $X$ be an algebraic variety with an action of an algebraic group $G$. A \emph{core} is a pair $(Y, H)$ where $Y\subseteq X$ is an algebraic subvariety and $H \subseteq G$ is an algebraic subgroup such that
\begin{itemize}
\item[(i)] $Y$ is orbitwise-closed for the $H$-action, i.e.\ the closure of the $H$-orbit satisfies $\overline{H\cdot y}\subseteq Y$, for all $y\in Y$. 
\item[(ii)] For every $x\in X$, we have $\overline{G\cdot x}\cap Y\neq \emptyset$. 
\item[(iii)] For every two $W_1, W_2\subseteq Y$ disjoint closed (in $Y$) $H$-invariant subsets, we have that $\overline{G\cdot W_1}\cap \overline{G\cdot W_2} = \emptyset$. 

\end{itemize}

\end{defn}

\begin{prop}\cite[Proposition 5.8]{gonzalez2024pseudo}
Suppose that the $G$-action on $X$ has a core $(Y, H)$, and there exists a pseudo-quotient $\pi:X\rightarrow \overline{X}$ for the $G$-action. Then $\pi$ restricts to a pseudo-quotient $\pi|_{Y}:Y\rightarrow \overline{X}$ for the $H$-action on $Y$. 
\end{prop}

\begin{cor}\label{cor:equality-pseudo-quotients}
In the hypotheses of the previous proposition, if $X$ admits a categorical quotient then, for any pseudo-quotient $X\rightarrow \overline{X}$ for the $G$-action on $X$ and any pseudo-quotient $Y \rightarrow \overline{Y}$ for the $H$-action on $Y$, we have
$[\overline{X}]=[\overline{Y}]$ in $\KVar$.
\end{cor}

\section{Root data of type ABCD groups}\label{sec:reductive-groups-general}

In this section, we shall briefly review some key results of the theory of reductive groups that will be needed in the following. For a detailed exposition of this rich and fascinating theory, please check \cite{conrad2020reductive}, \cite{borel2012linear} or \cite{humphreys2012linear}, among others.

\subsection{Root data of reductive algebraic groups}\label{sec:reductive-groups}

Let $G$ be an algebraic group over an algebraically closed field of characteristic zero (we shall typically think of $\CC$). Recall that $G$ is said to be \emph{reductive} if its unipotent radical is trivial. Most of the classical groups are reductive, such as the general linear group $\GL_n$, the special linear group $\SL_n$, the special orthogonal group $\SO_n$ or the symplectic group $\Sp_n$. In particular, a semisimple group, i.e.\ one whose solvable (or, equivalently, abelian) connected closed normal subgroups are trivial, is reductive.  A reductive group is automatically linear, meaning that it can be realized as a subgroup of $\GL_n$ for some $n$ or, equivalently, it admits a faithful linear representation. From now on, we shall suppose that $G$ is a reductive group.

A \emph{Borel} subgroup $B\subseteq G$ is a maximal Zariski closed and connected solvable algebraic subgroup. If we fix a faithful representation of $G$ on a vector space $V$, then the quotient $G/B$ can be identified with the space of full flags on $V$. For this reason, $G/B$ is usually known as the \emph{full-flag variety} and is a complete variety. In general, a subgroup $P\subseteq G$ is said to be \emph{parabolic} if $G/P$ is a complete variety or, equivalently, if $P$ contains a Borel subgroup. This way, a Borel subgroup $B$ is a minimal parabolic subgroup.

A \emph{maximal torus} $T \subseteq G$ is a connected abelian subgroup that is maximal with respect to the inclusion. Any two maximal tori are $G$-conjugate. Let us fix a maximal torus $T$ of $G$ and a Borel subgroup $B$ such that $T\subseteq B\subseteq G$. Each conjugacy class of parabolic subgroups $\mathcal{P}$ contains a unique representative $P$ satisfying $T\subseteq B\subseteq P\subseteq G$. We will call these subgroups $P$ the \emph{standard parabolic subgroups} (with respect to $T$ and $B$). 

Let us denote by $X^{\ast}(T)$ the \emph{lattice of characters} of $T$, 
\[X^{\ast}(T)=\{\chi:T\rightarrow \GG_m\},\]
where $\GG_m = \CC^*$ is the multiplicative group of non-zero scalars, and by $X_{\ast}(T)$ the \emph{lattice of cocharacters} of $T$ (or $1$-parameter subgroups)
\[X_{\ast}(T)=\{\lambda:\GG_m\rightarrow T\},\]
which are dual free abelian groups with duality integral pairing $(\chi, \lambda )\in \mathbb{Z}$ for $\chi \in X^\ast(T)$ and $\lambda \in X_\ast(T)$. We denote by $\Phi\subseteq X^{\ast}(T)$ the finite subset of roots, i.e.\ of characters of $T$ that arise as weights for the adjoint action of $T$ on the Lie algebra of $G$. These roots are in bijection with the set of coroots $\Phi^{\vee}\subseteq X_{\ast}(T)$ by $\Phi \ni \alpha\leftrightarrow \alpha^{\vee}\in \Phi^{\vee}$, and $(\alpha, \alpha^{\vee})=2$. For each root and coroot, define the reflections
\[s_{\alpha}:X^{\ast}(T)\rightarrow X^{\ast}(T)\;,\quad x\mapsto x-(x, \alpha^{\vee})\alpha,\]
\[s_{\alpha}^{\vee}:X_{\ast}(T)\rightarrow X_{\ast}(T)\;,\quad y\mapsto y-(\alpha, y)\alpha^{\vee}.\]
The \emph{Weyl group} $W$ of $G$ is the group of automorphisms of $X^{\ast}(T)$ generated by the reflections $s_{\alpha}$. This group is isomorphic to $N_G(T)/Z_G(T)$, where $N_G(H)$ denotes the \emph{normalizer} in $G$ of a subgroup $H \subseteq G$ and $Z_G(H)$ denotes its \emph{centralizer} in $G$. Recall that for a maximal torus $T$ we have that $Z_G(T) = T$.

In abstract terms, a \emph{root datum} is a tuple $R = (X^{\ast}, \Phi, X_{\ast}, \Phi^{\vee})$ satisfying the properties above. There is a bijection between root data and reductive algebraic groups. The dual root datum of $R=(X^{\ast}, \Phi, X_{\ast}, \Phi^{\vee})$ is the tuple $R^\vee = (X_{\ast}, \Phi^{\vee}, X^{\ast}, \Phi)$ that switches characters with cocharacters and roots with coroots. Given a reductive algebraic group $G$ with root datum $R$, its Langlands dual $^L G$ is the reductive group whose root datum is the dual $R^\vee$. 

The choice of a Borel $B$ containing a maximal torus $T$ (there are two opposite Borel subgroups whose intersection is exactly $T$) picks out one root from each pair $\pm \alpha \in \Phi$ and we denote by $\Phi^+$ the subset of \emph{positive roots} satisfying that there exists a cocharacter $\lambda\in X_{\ast}(T)$ with $(\alpha, \lambda)\neq 0$, for every $\alpha\in \Phi$, and such that $\Phi^+=\{\alpha\in \Phi : (\alpha, \lambda)>0\}.$
Two Borel subgroups $B_1, B_2\subseteq G$ are said to be adjacent if $\dim (B_1\cap B_2) = \dim B_1 - 1 = \dim B_2 -1$. The associated systems of positive roots, $\Phi_1^+$ and $\Phi_2^+$, are adjacent if  $\Phi_1^+\cap \Phi_2^+$ has one element less than each  $\Phi_i^+$, $i=1,2$. In this case, there exists a unique root $\alpha\in \Phi_1^+$ such that $s_{\alpha} \Phi_1^+=\Phi_2^+$. 

Given the system of positive roots $\Phi^+$, let $\Delta\subseteq \Phi^+$ be the subset of positive roots such that $\Phi^+$ and $s_{\alpha}\Phi^+$ are adjacent. %
We call $\Delta$ a set of \emph{simple roots} and it turns out that $\Delta$ is a basis of the lattice generated by the roots $\Phi$ and, hence, $\Delta^{\vee} = \{\alpha^\vee\mid \alpha \in \Delta\}$ is a basis of the lattice generated by the coroots $\Phi^{\vee}$. This $\Delta$ corresponds to the nodes of the Dynkin diagram of the semisimple part of the group $G$. The Weyl group $W$ is, indeed, generated by just these reflections $s_{\alpha}$ for $\alpha\in \Delta$.

Fixing a Borel subgroup $B\subseteq G$, there is an inclusion-preserving bijection between subsets $I \subseteq \Delta$ and standard parabolic subgroups. The standard parabolic subgroup associated to $I$ will be denoted by $P_I$. 
Indeed, associated to a subset $I\subseteq \Delta$ there is a cocharacter $\lambda_I : \GG_m \rightarrow T$ such that $(\alpha, \lambda_I)=0$, for every $\alpha\in I$. Then we define 
\begin{equation}\label{eq:parabolic}
    P_I=\left\{g\in G\; \left| \; \lim_{t \to 0} \lambda_I(t)g \lambda_I(t)^{-1} \; \text{exists} \right.\right\},
\end{equation}
which is a parabolic subgroup, and each standard parabolic subgroup of $G$ arises in this way (cf.\ \cite[Section 8.4]{springer1998groups}). Notice, in particular, that this characterization implies that, given an embedding $G \subseteq \GL_n$, every parabolic subgroup of $G$ is the restriction of a parabolic subgroup of $\GL_n$.

\begin{rem}
If we do not fix a Borel subgroup, then the bijection above is between subsets $I \subseteq \Delta$ and conjugacy classes of parabolic subgroups $\mathcal{P}_I$. However, choosing a Borel subgroup $B$ of $G$ selects a particular parabolic subgroup $P_I \in \mathcal{P}_I$, namely, the standard one.
\end{rem}

For each $I\subseteq \Delta$, let $\Phi_I\subseteq \Phi$ be the subset of roots generated by the elements $\alpha\in I$, and define $\Phi_I^{\vee}$ similarly.
Define
\begin{equation}\label{eq:torus-subset}
    T_I= \left(\bigcap_{\alpha\in I} \ker \alpha\right)^\circ,
\end{equation}
where the superscript $\circ$ denotes the connected component of the identity. Then $T_I$ is a subtorus of the maximal torus $T$. Notice that the relation $I \mapsto T_I$ is now inclusion-reversing. Let us denote by $L_I=Z_G(T_I)$ the centralizer of this torus, called the \emph{Levi subgroup} associated to $I$. 
This $L_I$ is a connected reductive subgroup of $G$ with Borel subgroup $B_I= B\cap L_I$. Reciprocally, we can also recover the torus $T_I$ associated to a Levi subgroup $L_I$ as the connected component of the identity of the center of $L_I$.
The Levi subgroup $L_{\{\alpha\}}$ associated to a simple root is the centralizer of $T_{\{\alpha\}} = (\ker\alpha)^\circ$. When $\alpha$ ranges in all roots $\Phi$, the subgroups $L_{\{\alpha\}}$, jointly with $T$, generate the whole group $G$, and ranging in $\Phi^{+}$ they do generate the Borel $B$. The subgroups $L_{\{\alpha\}}$, where $\alpha$ ranges in $\Phi_I$, generate $L_I$.

Finally, we shall also consider the \emph{Weyl group} of $I$, given by $W_I = N_G(T_I) / Z_G(T_I) = N_G(T_I) / L_I$. Note that the Weyl group $W_I$ does not coincide with the Weyl group of $L_I$, as Example \ref{ex:GL_n} shows.

\begin{rem}\label{rem:root-datum-levi}
    A remarkable property of the Levi subgroup $L_I \subseteq G$ associated to a subset of simple roots $I\subseteq \Delta$ is that it is characterized by the fact that its root datum is $(X^\ast(T), \Phi_I, X_\ast(T), \Phi_I^\vee)$, i.e.\ it has the same lattice of characters and cocharacters as $G$, but its root system is generated just by the roots of $I$. In particular, the Dynkin diagram of $L_I$ is the full sub-diagram of the Dynkin diagram of $G$ with only the vertices of $I$ and edges between them, with the same multiplicity and orientation as in $G$.
\end{rem}

\begin{example}\label{ex:subset-extreme}
The empty set $I=\emptyset$ corresponds to the Borel subgroup $P_\emptyset = B$, and the maximal torus $L_{\emptyset} = T_{\emptyset}=T$. On the other hand, the full $I=\Delta$ corresponds to $P_\Delta = G$, the torus is $T_{\Delta} = Z_G(G)^\circ$ the identity component of the center of $G$, and $L_{\Delta}=G$.
\end{example}

We characterize when two Levi subgroups are conjugate in the following lemma, which we will use in the proof of the main result in Theorem \ref{thm:main-result}.

\begin{lem}\label{lem:conjugate-levi-weil}
    Fix a Borel subgroup $B \subseteq G$ and let $L$ and $L'$ be two standard Levi subgroups. Then, $L$ and $L'$ are conjugate by $G$ if and only if they are conjugate by an element of the Weyl group of $G$.
\end{lem}

\begin{proof}
    Let $g \in G$ such that $gLg^{-1} = L'$, which implies that $gTg^{-1} = T'$ for the tori $T$ and $T'$ associated to the Levi subgroups $L$ and $L'$, respectively. Notice that, since both $L$ and $L'$ are standard Levi subgroups, the tori $T$ and $T'$ lie inside the same maximal torus $\tilde{T}$. Now, we obviously have $\tilde{T} \subseteq L$ since $L$ is the centralizer of $T \subseteq \tilde{T}$. But, furthermore, we also have $g^{-1}\tilde{T}g \subseteq L$ since for any $\tilde{t} \in \tilde{T}$ and $t \in T$, we have
    $$
        g^{-1}\tilde{t}gtg^{-1}\tilde{t}^{-1}g = g^{-1}gtg^{-1}\tilde{t}\tilde{t}^{-1}g = t.
    $$
    Here, in the first equality we used that $gtg^{-1} \in gTg^{-1} \subseteq \tilde{T}$, which is abelian.
    
    Therefore, both $\tilde{T}$ and $g^{-1}\tilde{T}g$ are tori of $L$, and they are maximal since $\tilde{T}$ is maximal in $G$, so there exists $\ell \in L$ such that $\ell g^{-1}\tilde{T}g\ell^{-1}=\tilde{T}$. Now, by definition, the element $g_0 := g \ell^{-1}$ normalizes $\tilde{T}$, but also satisfies that $g_0Tg_0^{-1} = g \ell^{-1}T \ell g^{-1} = gTg^{-1} = T'$ as $\ell$ commutes with $T$. Therefore, we also have $g_0Lg_0^{-1} = L'$ and thus the class of $g_0$ in the Weyl group provides the required element.
\end{proof}

\subsection{Root data for ABCD type groups}

There is a useful characterization of parabolic subgroups of a classical group $G$ as stabilizers of a flag of vector subspaces, possibly with extra conditions reflecting the structure of the group $G$. We will make precise this notion for the $A$, $B$, $C$ and $D$ series of the Dynkin diagram, through the classical four groups $\SL_n:=\SL(n, \CC)$, $\SO_{2n+1}:=\SO(2n+1,\CC)$, $\Sp_{2n}:=\Sp(2n,\CC)$ and $\SO_{2n}:=\SO(2n,\CC)$, whose Lie algebras correspond to these Dynkin series. We will start, though, by considering the reductive non-simple type $A$ group $\GL_n := \GL(n, \CC)$ and the adjoint-type $A$ group $\PGL_n := \PGL(n, \CC)$, which is the Langlands dual of $\SL_n$. 

\subsubsection{$G=\GL_n$, reductive non-simple,  with Dynkin diagram $A_{n-1}$}\label{sec:lie-gl}

\quad

Fixing a maximal torus $T$ of $\GL_n$ means fixing a basis of $\mathbb{C}^n$ such that $T$ are the diagonal matrices $T=\{ {\rm diag}(a_1, \ldots, a_n) \,\mid a_i\neq 0\}$; in this way elements of $\GL_n$ correspond to $n\times n$ non-singular matrices. Fixing a Borel $B\subseteq \GL_n$ containing $T$ is to choose between upper/lower triangular matrices in the given basis: let us fix $B$ to be the upper-triangular matrices. 

The lattice of characters of $T$ is 
\[X^{\ast}(T)=\{\chi:T\rightarrow \GG_m\;, \quad {\rm diag}(a_1, \ldots, a_n)\mapsto \chi(a_1, \ldots, a_n)\},\]
which is generated by the characters $\chi_i(a_1, \ldots, a_n) = a_i$ for $1 \leq i \leq n$. Therefore, $X^\ast(T) \cong \mathbb{Z}^n$.
In the same vein, the dual lattice of cocharacters is 
\[X_{\ast}(T)=\{\lambda:\GG_m \rightarrow T\;, \quad t\mapsto {\rm diag}(t^{\lambda_1},\ldots, t^{\lambda_n})\}\cong \mathbb{Z}^n\]
The pairing is given by $(\chi, \lambda)=n\in\mathbb{Z}$ if $(\chi\circ \lambda)(t)=t^n$.

The roots $\Phi=\{\alpha_{ij}:1\leq i,j\leq n, i\neq j\}$ are given by the characters
\[\alpha_{ij}:T\rightarrow \GG_m\;, \quad {\rm diag}(a_1, \ldots, a_n)\mapsto a_i\cdot a_j^{-1}.\]
The coroots $\Phi^{\vee}=\{\alpha^{\vee}_{ij}:1\leq i,j\leq n, i\neq j\}$ are the cocharacters 
\[\alpha^{\vee}_{ij}:\GG_m \rightarrow T\;, \quad t\mapsto {\rm diag}(1,\ldots, 1,t,1,\ldots,1,t^{-1}, 1, \ldots, 1)\]
with the entry $t$ in the position $i$ and the entry $t^{-1}$ in the position $j$. Note that this verifies $(\alpha_{ij}, \alpha^{\vee}_{ij})=2$. Under the isomorphism between $X_{\ast}(T)$, $X^{\ast}(T)$ and $\mathbb{Z}^n$, the roots $\alpha_{ij}$ and coroots $\alpha^{\vee}_{ij}$ correspond to vectors $e_i-e_j$, where $e_1,\ldots, e_n$ is the canonical basis of $\mathbb{Z}^n$, the pairing being the standard inner product in $\mathbb{Z}^n$. 

The choice of the Borel subgroup being the upper triangular matrices yields that the positive roots $\Phi^{+}$ are those $\alpha_{ij}$ with $i<j$. A cocharacter $\lambda\in X_{\ast}(T)$ such that $(\alpha_{ij},\lambda)>0$ if $i<j$ is 
\[\lambda:\GG_m \rightarrow T\;, \quad t\mapsto {\rm diag}(t^{n},t^{n-1},\ldots, t^2, t).\]
Associated with each root $\alpha_{ij}$, the reflection $s_{\alpha_{ij}}$ acts as
\[s_{\alpha_{ij}}(a_1,\ldots, a_i, \ldots, a_j, \ldots, a_n)=(a_1,\ldots, a_j, \ldots, a_i, \ldots, a_n),\]
interchanging the entries $(i,j)$. From this we get that the basis of simple roots is given by $\Delta=\{\alpha_{12}, \alpha_{23}, \ldots, \alpha_{(n-1) n}\}$. The Weyl group $W$ of $\GL_n$ is the group generated by the reflections $s_{\alpha_{ij}}$, which can be interpreted as transpositions $(i,j)$, $i\neq j$, $1\leq i,j\leq n$ generating the symmetric group on $n$ elements $S_n$. Note that $W$ is generated by those transpositions of the form $\alpha_{i(i+1)}$, which are the ones in $\Delta$. There are $2^{|\Delta|}=2^{n-1}$ parabolic subgroups (up to conjugation).

\subsubsection{$G=\SL_n$, simple and simply connected, with Dynkin diagram $A_{n-1}$.}\label{sec:lie-sl}

\quad

Fixing a maximal torus $T_{\SL}$ of $\SL_n$ means fixing a basis of $\mathbb{C}^n$ such that $T_{\SL}$ are the diagonal matrices, so we can take $T_{\SL}=\{ {\rm diag}(a_1, \ldots, a_n) \,\mid a_1\cdots a_n=1\}$. Fixing a Borel $B_{\SL}\subseteq \SL_{n}$ again amounts to choosing between upper/lower triangular matrices and we fix $B_{\SL}$ to be the upper-triangular matrices. 

The lattice of characters of $T_{\SL}$ is now
\[X^{\ast}(T_{\SL})=\{\chi:T_{\SL}\rightarrow \GG_m\;, \quad {\rm diag}(a_1, \ldots, a_n)\mapsto \chi(a_1, \ldots, a_n)\},\]
generated by the characters $\chi_i(a_1, \ldots, a_n) = a_i$ for $1 \leq i \leq n$, which are now dependent since $\chi_1\chi_2\cdots \chi_n = 1$. In other words, the map $\ZZ^n \to X^\ast(T_{\SL})$, given by $(x_1, \ldots, x_n) \mapsto \chi_1^{x_1}\cdots\chi_n^{x_n}$, has kernel equal to the span of $(1, \ldots, 1) \in \ZZ^n$, so we have an isomorphism
$$
    X^\ast(T_{\SL}) \cong \left\{(x_1, \ldots, x_n) \in \ZZ^n \right\}/(1, \ldots, 1) \cong \ZZ^{n-1}.
$$
The dual lattice of cocharacters is
\[X_{\ast}(T_{\SL})=\{\lambda:\GG_m \rightarrow T_{\SL}\;, \quad t\mapsto {\rm diag}(t^{\lambda_1},\ldots, t^{\lambda_n})\}\cong \{(\lambda_1, \ldots, \lambda_n) \in \ZZ^n \,\mid\, \lambda_1 +\cdots + \lambda_n = 0 \} \cong \mathbb{Z}^{n-1},\]
where the pairing is the same as in $\GL_n$.

Roots and coroots of $\SL_n$ are also the same as in $\GL_n$.
Observe that now $\Phi$ and $\Phi^{\vee}$ rationally span $X^\ast(T_{\SL})$ and $X_\ast(T_{\SL})$, respectively. Positive and negative roots, then, remain the same as in $\GL_n$, and the basis of simple roots is $\Delta=\{\alpha_{12}, \alpha_{23}, \ldots, \alpha_{(n-1) n}\}$, with Weyl group generated by the reflections $s_{\alpha_{ij}}$, equalling the symmetric group $S_n$ in $n$ elements.

\subsubsection{$G=\PGL_n$, simple of adjoint type, with Dynkin diagram $A_{n-1}$}\label{sec:lie-pgl}

\quad 

Again, fixing a maximal torus $T_{\PGL}$ of $\PGL_n$ is the same as fixing a basis of $\mathbb{C}^n$ such that $T_{\PGL}=\{ {\rm diag}(a_1, \ldots, a_n) \,\mid a_i\neq 0\}/\GG_m$ are the diagonal matrices, and choose the Borel subgroup $B_{\PGL}$ to be the upper-triangular matrices up to scalar. 

The group $\PGL_n$ is the Langlands dual of $\SL_n$, then its root datum is given by interchanging that of $\SL_n$. In particular, the lattice of characters is \[X^{\ast}(T_{\PGL})=\{
\chi:T_{\PGL}\rightarrow \GG_m\;, \quad \left[{\rm diag}(a_1, \ldots, a_n)\right]\mapsto \chi(a_1, \ldots, a_n)\}\]
where the map $\ZZ^n \ni (x_1, \ldots, x_n) \mapsto \chi_1^{x_1} \ldots \chi_n^{x_n} $ is well-defined in $T_{\PGL}$ only if the resulting character has degree zero. Hence, we have an identification
\[X^{\ast}(T_{\PGL}) \cong \{(x_1, \ldots, x_n) \in \ZZ^n\,\mid\, x_1 + \cdots + x_n = 0\} = X_{\ast}(T_{\SL}) \cong \ZZ^{n-1}.
\] 
The dual lattice of cocharacters is
\[X_{\ast}(T_{\PGL})=\left\{\lambda:\GG_m \rightarrow T_{\PGL}\;, \quad t\mapsto \left[{\rm diag}(t^{\lambda_1},\ldots, t^{\lambda_n})\right]\right\}\cong X^{\ast}(T)/{\rm diag}(t,\ldots,t) = X^{\ast}(T_{\SL}) \cong \mathbb{Z}^{n-1}.\]
The pairing is the same as in $\GL_n$  and the roots and coroots of $\PGL_n$ are interchanged from those of $\SL_n$.

\subsubsection{$G=\Sp_{2n}$, simple with Dynkin diagram $C_{n}$}\label{sec:lie-sp}

\quad 

Let $(V, \omega)$ be a symplectic complex vector space, which is the data of a complex vector space $V$ of dimension $2n$ together with a complex symplectic bilinear form $\omega: V\times V\rightarrow \mathbb{C}$, which is bilinear, alternating ($\omega(v,v)=0$ for all $v\in V$) and non-degenerate ($\omega(u,v)=0$ for all $v\in V$ implies $u=0$). The group $\Sp_{2n}$ is the group of automorphisms $A: V \to V$ such that $\omega(Au, Av) = \omega(u,v)$ for all $u,v \in V$ or, equivalently, such that $A\omega A^t = \omega$.

A Borel subgroup of $\Sp_{2n}$ is the intersection of a Borel subgroup of $\GL_{2n}$ with $\Sp_{2n}$. Then we can assume that fixing a Borel in $\Sp_{2n}$ is equivalent to fixing a basis of $\CC^{2n}$ of the form $(x_1, x_2,\ldots, x_n,y_1, y_2,\ldots, y_n)$ of $V$ such that, in that basis, the symplectic form $\omega$ adopts its standard form
$$
    \omega = \left(
\begin{array}{cc}
0 & I_n\\
-I_n & 0
\end{array}
\right),
$$
i.e.\ $\omega(x_i, x_j)=\omega(y_i, y_j)=0$ and $\omega(x_i, y_j)=-\omega(y_j, x_i)=\delta_{ij}$. 

In terms of root data, a maximal torus of $\Sp_{2n}$ is given by diagonal matrices of the form $T=\{ {\rm diag}(a_1, \ldots, a_{n}, a_1^{-1},\ldots, a_n^{-1})\,\mid a_i\neq 0\}$ (i.e.\ it identifies with a maximal torus of $\GL_n$).
The lattice of characters of $T$ is 
\[X^{\ast}(T)=\{\chi:T\rightarrow \GG_m\;, \quad {\rm diag}(a_1, \ldots, a_{n}, a_1^{-1},\ldots, a_n^{-1})\mapsto \chi(a_1, \ldots, a_{n})\}\cong \mathbb{Z}^n,\]
with dual lattice of cocharacters 
\[X_{\ast}(T)=\{\lambda:\GG_m \rightarrow T\;, \quad t\mapsto {\rm diag}(t^{\lambda_1},\ldots, t^{\lambda_n}, t^{-\lambda_{1}}, \ldots, t^{-\lambda_n})\}\cong \mathbb{Z}^n\]
where pairing is given by $(\chi, \lambda)=n\in\mathbb{Z}$ if $(\chi\circ \lambda)(t)=t^n$.

The roots are $\Phi=\{\alpha_{i, \epsilon}:{1\leq i\leq n, \epsilon=\pm 1}\}\cup \{\beta_{ij, \delta \kappa} : {1\leq i,j\leq n, i\neq j, \delta, \kappa=\pm 1}\}$ and are given by the characters
\[\alpha_{i, \epsilon}:T\rightarrow \GG_m\;, \quad {\rm diag}(a_1, \ldots, a_{n}, a_1^{-1},\ldots, a_n^{-1})\mapsto a_i^{2 \epsilon},\]
\[\beta_{ij, \delta \kappa}:T\rightarrow \GG_m\;, \quad {\rm diag}(a_1, \ldots, a_{n}, a_1^{-1},\ldots, a_n^{-1})\mapsto a_i^{\delta}\cdot a_j^{\kappa}.\]
The coroots $\Phi^{\vee}=\{\alpha_{i, \epsilon}^{\vee}:{1\leq i\leq n, \epsilon=\pm 1}\}\cup \{\beta^{\vee}_{ij, \delta\kappa}:{1\leq i,j\leq n, i\neq j, \delta,\kappa=\pm 1}\}$ are the cocharacters 
\[\alpha^{\vee}_{i,\epsilon}:\GG_m \rightarrow T\;, \quad t\mapsto {\rm diag}(1,\ldots, 1,t^{\epsilon},1,\ldots, t^{-\epsilon},1,\ldots,1),\]
with the entry $t^{\epsilon}$ in the position $i$ and the entry $t^{-\epsilon}$ is in the position $n+i$, and 
\[\beta^{\vee}_{ij, \delta\kappa}:\GG_m \rightarrow T\;, \quad t\mapsto {\rm diag}(1,\ldots, 1,t^{\delta},1,\ldots,1,t^{\kappa}, 1, \ldots, 1, t^{-\delta},1,\ldots,1,t^{-\kappa}, 1, \ldots, 1),\]
with the entry $t^{\delta}$ in the position $i$ the entry $t^{-\delta}$ is in the position $n+i$, the entry $t^{\kappa}$ in the position $j$ and the entry $t^{-\kappa}$ is in the position $n+j$. 
Note that this verifies $(\alpha_{i,\epsilon}, \alpha^{\vee}_{i,\epsilon})=2$ and $(\beta_{ij,\delta\kappa}, \beta^{\vee}_{ij,\delta\kappa})=2$. Under the isomorphism between $X_{\ast}(T)$, $X^{\ast}(T)$ and $\mathbb{Z}^n$, the roots $\alpha_{i,\epsilon}$ and $\beta_{ij,\delta\kappa}$ relate to vectors $2\epsilon e_i$ and $\delta e_i + \kappa e_j$, where $(e_1,\ldots, e_n)$ is the canonical basis of $\mathbb{Z}^n$, and the coroots $\alpha_{i,\epsilon}^{\vee}$ correspond to $\epsilon e_i$, the pairing being the standard inner product in $\mathbb{Z}^n$. 

The Weyl group of $\Sp_{2n}$ is $W= \ZZ_2^{n}\rtimes S_n$, where $S_n$ is generated by reflections coming from roots $\beta_{ij,\delta\kappa}$, yielding permutation of blocks between coordinates  from $1$ to $n$ (which give the same permutation in coordinates from $n+1$ to $2n$), and each factor $\ZZ_2$ acts by switching the coordinate $a_i$ with $a_i^{-1}$ in the diagonal. In terms of the basis $(x_1, \ldots, x_n, y_1, \ldots, y_n)$ this has the effect of permuting blocks among the $x$'s coordinates (and then permuting the $y$'s coordinates accordingly), and switching coordinate $x_i$ with $y_i$. 

A basis $\Delta$ of simple roots is given by $\Delta=\{\beta_{12,+ -}, \beta_{23,+-}, \beta_{34,+-},\ldots, \beta_{(n-1) n, +-}, \alpha_{n, +}\}$, the last one being the unique long root in the Dynkin diagram. There are $2^{|\Delta|}=2^n$ parabolic subgroups (up to conjugation). 

\begin{rem}
    As an example of sporadic isomorphisms in low dimension, we have that $\Sp_{2} = \SL_2$. This can be seen from the fact that both spaces have maximal torus $T=\CC^*$, lattice of characters equal to $X^\ast(T)=\ZZ$ and with the unique root $\alpha(a) = a^2$.
\end{rem}

\subsubsection{$G=\SO_{2n+1}$, simple with Dynkin diagram $B_{n}$.}\label{sec:lie-so-odd}

\quad

Let $V$ be a $(2n+1)$-dimensional complex vector space. We consider a non-degenerate symmetric bilinear form $Q: V \times V \to \CC$. Then $\SO_{2n+1}$ is the group of linear automorphisms $A: V \to V$ such that $Q(Av, Av) = Q(v,v)$ for all $v \in V$. After a change of basis, we can suppose that $V = \CC^{2n+1}$ with associated quadratic form
$$
    Q(a) = a_1a_{n+2} + a_2a_{n+3} + \cdots + a_{n}a_{2n+1} + a_{n+1}^2,
$$
where $a = (a_1, \ldots, a_{2n+1})$.

A maximal torus is given by diagonal matrices $T=\{ {\rm diag}(a_1, \ldots, a_{n},1, a_1^{-1},\ldots, a_n^{-1}) \,\mid a_i\neq 0\}$. 
We fix a Borel $B\subseteq \SO_{2n+1}$ by choosing the upper-triangular orthogonal matrices. The group $\SO_{2n+1}$ is the Langlands dual of $\Sp_{2n}$, therefore its root datum is recovered from that of the symplectic group by interchanging characters and roots by cocharacters and coroots. We include here the full description for completeness. 

The lattice of characters of $T$ is 
\[X^{\ast}(T)=\{\chi:T\rightarrow \GG_m\;, \quad {\rm diag}(a_1, \ldots, a_{n}, 1, a_1^{-1},\ldots, a_n^{-1})\mapsto \chi(a_1, \ldots, a_{n})\}\cong \mathbb{Z}^n\]
and the lattice of cocharacters is
\[X_{\ast}(T)=\{\lambda:\GG_m \rightarrow T\;, \quad t\mapsto {\rm diag}(t^{\lambda_1},\ldots, t^{\lambda_n}, 1, t^{-\lambda_{1}}, \ldots, t^{-\lambda_n})\}\cong \mathbb{Z}^n.\]
The roots are $\Phi=\{\alpha_{i, \epsilon} : {1\leq i\leq n, \epsilon=\pm 1}\}\cup \{\beta_{ij, \delta \kappa}:{1\leq i,j\leq n, i\neq j, \delta, \kappa=\pm 1}\}$ with
\[\alpha_{i, \epsilon}:T\rightarrow \GG_m\;, \quad {\rm diag}(a_1, \ldots, a_{n}, 1, a_1^{-1},\ldots, a_n^{-1})\mapsto a_i^{\epsilon},\]
\[\beta_{ij, \delta \kappa}:T\rightarrow \GG_m\;, \quad {\rm diag} (a_1, \ldots, a_{n}, 1, a_1^{-1},\ldots, a_n^{-1})\mapsto a_i^{\delta}\cdot a_j^{\kappa}, \]
and the coroots are $\Phi^{\vee}=\{\alpha_{i, \epsilon}^{\vee}:{1\leq i\leq n, \epsilon=\pm 1}\}\cup \{\beta^{\vee}_{ij, \delta\kappa}:{1\leq i,j\leq n, i\neq j, \delta,\kappa=\pm 1}\}$ with
\[\alpha^{\vee}_{i,\epsilon}:\GG_m \rightarrow T\;, \quad t\mapsto {\rm diag}(1,\ldots, 1,t^{2\epsilon},1,\ldots, t^{-2\epsilon},1,\ldots,1),\]
\[\beta^{\vee}_{ij, \delta\kappa}:\GG_m \rightarrow T\;, \quad t\mapsto {\rm diag}(1,\ldots, 1,t^{\delta},1,\ldots,1,t^{\kappa}, 1, \ldots, 1, t^{-\delta},1,\ldots,1,t^{-\kappa}, 1, \ldots, 1).\]

The Weyl group of $\SO_{2n+1}$ is $W= \ZZ_2^{n}\rtimes S_n$, where $S_n$ is generated by reflections coming from roots $\beta_{ij,\delta\kappa}$, yielding permutation of blocks between coordinates from $1$ to $n$ (which give the same permutation in coordinates from $n+2$ to $2n+1$), and each $\ZZ_2$ acting switching the coordinate $a_i$ with $a_i^{-1}$ in the diagonal. A basis $\Delta$ of simple roots is given by $\Delta=\{\beta_{12,+ -}, \beta_{23,+-}, \beta_{34,+-},\ldots, \beta_{(n-1) n, +-}, \alpha_{n, +}\}$, the last one being the unique short root in the Dynkin diagram, and there are $2^{|\Delta|}=2^n$ parabolic subgroups (up to conjugation).

\subsubsection{$G=\SO_{2n}$, simple with Dynkin diagram $D_{n}$}\label{sec:lie-so-even}

\quad

With the same notations as for the odd orthogonal group $\SO_{2n+1}$, we take $\CC^{2n}$ the $2n$-dimensional complex vector space with a chosen basis, and we equip it with the quadratic form
$$
    Q(a) = a_1a_{n+1} + a_2a_{n+2} + \cdots + a_{n}a_{2n}
$$
where $a = (a_1, \ldots, a_{2n})$. In this way, $\SO_{2n}$ is the group of automorphisms of $\CC^{2n}$ that preserve $Q$. A maximal torus is given by diagonal matrices $T=\{ {\rm diag}( a_1, \ldots, a_{n}, a_1^{-1},\ldots, a_n^{-1}) \,\mid a_i\neq 0\}$ and we fix a Borel subgroup $B\subseteq \SO_{2n}$ by choosing the upper-triangular orthogonal matrices. 

The lattices of characters and cocharacters of $T$ are now
\[X^{\ast}(T)=\{\chi:T\rightarrow \GG_m\;, \quad {\rm diag}(a_1, \ldots, a_{n}, a_1^{-1},\ldots, a_n^{-1})\mapsto \chi(a_1, \ldots, a_{n})\}\cong \mathbb{Z}^n\]
\[X_{\ast}(T)=\{\lambda:\GG_m \rightarrow T\;, \quad t\mapsto {\rm diag}( t^{\lambda_1},\ldots, t^{\lambda_n}, t^{-\lambda_{1}}, \ldots, t^{-\lambda_n})\}\cong \mathbb{Z}^n\]
and the roots are $\Phi=\{\beta_{ij, \delta \kappa}:{1\leq i,j\leq n, i\neq j, \delta, \kappa=\pm 1}\}$ given by the characters
\[\beta_{ij, \delta \kappa}:T\rightarrow \GG_m\;, \quad {\rm diag}(a_1, \ldots, a_{n}, a_1^{-1},\ldots, a_n^{-1})\mapsto a_i^{\delta}\cdot a_j^{\kappa},\]
coroots being $\Phi^{\vee}=\{\beta^{\vee}_{ij, \delta\kappa}:{1\leq i,j\leq n, i\neq j, \delta,\kappa=\pm 1}\}$
\[\beta^{\vee}_{ij, \delta\kappa}:\GG_m \rightarrow T\;, \quad t\mapsto {\rm diag}(1,\ldots, 1,t^{\delta},1,\ldots,1,t^{\kappa}, 1, \ldots, 1, t^{-\delta},1,\ldots,1,t^{-\kappa}, 1, \ldots, 1).\]
Here the entry $t^{\delta}$ lies in the position $i$, the entry $t^{-\delta}$ is in the position $n+i$, the entry $t^{\kappa}$ in the position $j$ and the entry $t^{-\kappa}$ is in the position $n+j$, verifying $(\beta_{ij,\delta\kappa}, \beta^{\vee}_{ij,\delta\kappa})=2$. Under the isomorphism between $X_{\ast}(T)$, $X^{\ast}(T)$ and $\mathbb{Z}^n$, the roots and coroots $\beta_{ij,\delta\kappa}$ and $\beta_{ij,\delta\kappa}^{\vee}$ relate to vectors $\delta e_i + \kappa e_j$, where $(e_1,\ldots, e_n)$ is the canonical basis of $\mathbb{Z}^n$.

Let $H_n$ be the kernel of the morphism $\ZZ_2^{n}\rightarrow \ZZ_2$ given by $(\epsilon_1, \cdots , \epsilon_n) \mapsto \epsilon_1\cdots \epsilon_n$, that is $H_{n}$ is the subgroup of $\ZZ_2^n$ with an even number of non-identity elements. The Weyl group in this case is $W= H_{n} \rtimes S_n$.
A basis $\Delta$ of simple roots is given by $\Delta=\{\beta_{12,+ -}, \beta_{23,+-},\ldots, \beta_{(n-2)(n-1), +-}, \beta_{(n-1) n, +-}, \beta_{(n-1) n, ++}\}$, the last two being the branching in the Dynkin diagram (note that we have product of roots given by $(\beta_{(n-2)(n-1), +-},\beta_{(n-1) n, +-})=(\beta_{(n-2)(n-1), +-}, \beta_{(n-1) n, ++})=-1$, meaning that these are dots connected by an edge in the diagram, but last two roots are not because of $(\beta_{(n-1) n, +-}, \beta_{(n-1) n, ++})=0$). There are $2^{|\Delta|}=2^n$ parabolic subgroups (up to conjugation).

\section{Parabolic stratification of $G$-character varieties}\label{sec:parabolic-stratification}

Let us fix a reductive group $G$ over an algebraically closed field of characteristic zero, and a finitely generated group $\Gamma$. Let us consider the representation variety $\mathcal{R}_G(\Gamma) = \Hom(\Gamma, G)$. To lighten the notation, throughout this section we shall remove the reference to $\Gamma$ when it is clear from the context, and we shall denote the representation variety just by $\mathcal{R}_G$.

\begin{defn}
The GIT quotient    
\[
\mathcal{X}_G(\Gamma) =\mathcal{R}_G(\Gamma)\quot G
\]
with respect to the conjugation action of $G$ on the representation variety $\mathcal{R}_G(\Gamma)$ is an algebraic variety known as the \emph{$G$-character variety} of $\Gamma$ .
\end{defn}

A representation $\rho: \Gamma \to G$ of $G$ is said to be \emph{reducible} if there exists a proper parabolic subgroup $P \subsetneq G$ such that $\rho(\Gamma) \subseteq P$. Otherwise, $\rho$ is said to be \emph{irreducible}. The set of irreducible representations forms an open set $\mathcal{R}_{G}^{*}(\Gamma)\subseteq\mathcal{R}_{G}(\Gamma)$ and the corresponding quotient $\mathcal{X}_{G}^{*}(\Gamma)\subseteq\mathcal{X}_{G}(\Gamma)$ consists of closed orbits. However, the situation for reducible representations is more complicated since their orbits may not be closed and thus representations that are not conjugate may be identified in the GIT quotient.

The idea of this section will be to decompose this GIT quotient into a locally closed stratification given by pseudo-quotients as in Section \ref{ssec:pseudo_quotients}. The strategy will be to partition the representation variety into invariant subsets associated with parabolic subgroups of $G$, indexed by subsets of simple roots. Then it occurs that the restriction of the $G$-conjugacy action to each of these pieces we study has a core as in Section \ref{ssec:cores}, for certain subvariety and certain subgroup, which are carefully identified in terms of root data. This way, we obtain a motivic decomposition of the $G$-character variety in terms of these cores, capturing all of its topology in Theorem \ref{thm:main-result}.

Given a parabolic subgroup $P \subseteq G$, we will say that a representation $\rho: \Gamma \to G$ \emph{preserves $P$} if $\rho(\gamma) \in P$ for all $\gamma \in \Gamma$ or, equivalently, if it defines a representation $\rho: \Gamma \to P$. Given $I \subseteq \Delta$ a subset of the set of simple roots of $G$, we will denote by $\mathcal{P}_I$ the set of parabolic subgroups conjugate to the standard parabolic subgroup $P_I$ determined by $I$. We will say that $\rho: \Gamma \to G$ is \emph{of type $I$} if $\rho$ preserves $P$ for some $P \in \mathcal{P}_I$, not necessarily standard. Let us denote by $\widehat{\mathcal{R}}_{{P}_I}$ the set of representations of type $I$ and notice that
$$
    \widehat{\mathcal{R}}_{{P}_I} = \bigcup_{P \in \mathcal{P}_I} \mathcal{R}_{P} = G \cdot {\mathcal{R}}_{P_I}
$$
for the conjugacy action of $G$.

\begin{prop}
For any subset $I \subseteq \Delta$, the set $\widehat{\mathcal{R}}_{{P}_I}$ is a locally closed subvariety of the representation variety $\mathcal{R}_G$.

\begin{proof}
Fix a Borel subgroup $B$ of $G$. Hence, there is a distinguished parabolic subgroup $P_I \in \mathcal{P}_I$ that contains $B$, the standard one. The set of representations preserving $P_I$, $\mathcal{R}_{P_I} \subseteq \mathcal{R}_G$, is a closed subset since $P_I \subseteq G$ is a closed subgroup. Now, we have that $\widehat{\mathcal{R}}_{{P}_I}$ is the collection of $G$-orbits of $\mathcal{R}_{P_I}$, and thus it is a locally closed set.
\end{proof}
\end{prop}

From now on, we shall fix a Borel subgroup $B \subseteq G$ which in particular determines a maximal torus $T \subseteq B$ and a standard parabolic subgroup $ P_I \supseteq B$ for any $I\subseteq \Delta$. Now, consider ${\mathcal{R}}_{P_I}^\star$ the collection of representations $\rho: \Gamma \to P_I$ that are irreducible as $P_I$-representations, meaning that there exists no parabolic subgroup $P$ of $G$ properly contained in $P_I$ such that $\rho(\Gamma) \subseteq P$. In other words,
$$
    {\mathcal{R}}_{P_I}^\star = {\mathcal{R}}_{P_I} - \bigcup_{I' \subsetneq I} {\mathcal{R}}_{{P}_{I'}}.
$$

\begin{rem}\label{rmk:parabolic-of-parabolic}
    As pointed out by the anonymous referee, the parabolic group $P_I$ is not in general reductive. However, any parabolic subgroups $P$ and $Q$ of a reductive group $G$ over an algebraically closed field such that $Q \subsetneq P$ satisfy that $P/Q$ is projective. Therefore, the notion of \textit{parabolic subgroup of $P_I$} is still meaningful in our context.
    
    For completeness, we include here the proof of this fact that was provided to us by the referee. Consider the Levi decomposition of $P$, given by $P = L \ltimes U_P$, where $L$ is a Levi subgroup and $U_P$ is the unipotent radical of $P$. Since $Q \subsetneq P$
is a parabolic subgroup, it must contain the unipotent radical $U_P$.
Therefore, $Q$ can be decomposed as
\[
Q = (Q \cap L) \ltimes U_P.
\]
The intersection $Q \cap L$ is a parabolic subgroup of the reductive
group $L$. By the isomorphism theorems for algebraic groups, the
quotient space $P/Q$ is isomorphic to
\[
P/Q = (L \ltimes U_P)/((Q \cap L) \ltimes U_P)
\cong L/(Q \cap L).
\]
Because $Q \cap L$ is a parabolic subgroup of $L$, the quotient
$L/(Q \cap L)$ is a standard projective flag variety. Consequently,
$P/Q$ is a projective variety. 
\end{rem}

We also consider $\widehat{\mathcal{R}}_{{P}_I}^\star$ as the collection of representations conjugate to one of ${\mathcal{R}}_{P_I}^\star$, that is
$$
    \widehat{\mathcal{R}}_{\mathcal{P}_I}^\star = G\cdot {\mathcal{R}}_{P_I}^\star
$$

Notice that each stratum ${\mathcal{R}}_{\widehat{P}_I}^\star$ is invariant under conjugation.
These spaces induce a decomposition (not necessarily disjoint) of the representation variety by subsets of the Dynkin diagram of the form
\begin{equation}\label{eq:decomposition-representation}
    \mathcal{R}_G = \bigcup_{I \subseteq \Delta} \widehat{\mathcal{R}}_{{P}_I}^\star.
\end{equation}
Furthermore, observe that each stratum $\widehat{\mathcal{R}}_{{P}_I}^\star$ is orbitwise-closed i.e.\ the Zariski closure of the orbit of any element of $\widehat{\mathcal{R}}_{{P}_I}^\star$ is contained in $\widehat{\mathcal{R}}_{{P}_I}^\star$. 

Recall that, as explained in Section \ref{sec:reductive-groups}, associated to a subset $I \subseteq \Delta$, we also have a unique torus $T_I \subseteq G$ contained in the Borel subgroup $B$, as given in (\ref{eq:torus-subset}). Its centralizer $L_I = Z_G(T_I)$ is the associated Levi subgroup of $P_I$. 
Using it, we can consider the {Weyl group} of $I$ as
$$
    W_I = N_I / L_I,
$$
where $N_I := N_G(T_I)$ denotes the normalizer of $T_I$ in $G$. Observe that $N_I = N_G(L_I)$ since $N_G(T_I) = N_G(Z_G(T_I)) = N_G(L_I)$. 

\begin{rem}\label{rmk:emptyset}
As discussed in Example \ref{ex:subset-extreme}, in the case $I = \emptyset$ we have $T_\emptyset = L_\emptyset = T$, the maximal torus of $G$ contained in the Borel subgroup. In particular, we have that $W_\emptyset = N_G(T) / T$ is the usual Weyl group of $G$.
\end{rem}

\begin{prop}\label{lem:orbit-parabolic}
Let $x \in P$, for a parabolic subgroup $P$. Then the Zariski closure of the conjugacy orbit of $x$ in $G$ intersects the Levi subgroup of $P$.
\end{prop}

\begin{proof}
Without loss of generality, we can suppose that $P=P_I$ for $I\subseteq \Delta$ and let $\lambda_I$ be the cocharacter associated to $P_I$. The element $y:=\underset{t\rightarrow 0}{\lim}\, \lambda_I(t)x \lambda_I(t)^{-1}$ exists by definition of $P_I$, as used in (\ref{eq:parabolic}), and $y\in \overline{G\cdot x}$ by the way it is obtained, where $G\cdot x$ denotes the $G$-conjugacy orbit of $x$ and $\overline{G\cdot x}$ its closure. We see that this $y$ belongs to $L_I$. Indeed, for all $s \in \GG_m$, we have
\[
 \lambda_I(s) y \lambda_I(s)^{-1} = \lim_{t\to 0} \lambda_I(s) \lambda_I(t) x\lambda_I(t)^{-1}\lambda_I(s)^{-1} =
\lim_{t\to 0} \lambda_I(st) x \lambda_I(st)^{-1} =y
.\]
Hence, $y\in \overline{G\cdot x}\cap L_I$.
\end{proof}

Let us come back to the representation variety. Given a subset $I \subseteq \Delta$, we can consider the variety $\mathcal{R}_{L_I}^\star$ of irreducible representations on $L_I$. Again, irreducibility is considered in $L_I$: a representation $\rho: \Gamma \to L_I$ is in $\mathcal{R}_{L_I}^\star$ if there exists no proper parabolic subgroup $P$ of $L_I$ with $\rho(\Gamma) \subseteq P$.
Then, as an application of Proposition \ref{lem:orbit-parabolic}, we get the following result.

\begin{cor}\label{cor:intersection-core}
If $\rho \in \widehat{\mathcal{R}}_{{P}_I}^\star$, then the Zariski closure of the $G$-conjugacy orbit of $\rho$ intersects $\mathcal{R}_{L_I}^\star$.
\end{cor}

\begin{proof}
    Let $\rho \in \widehat{\mathcal{R}}_{{P}_I}^\star$ so that $g\rho g^{-1} \in \mathcal{R}_{{P}_I}^\star$ for a certain $g \in G$. By applying Proposition \ref{lem:orbit-parabolic} to the image of each of the generators of $\Gamma$, we get that $g\rho g^{-1}$, and thus $\rho$, has an element $\rho' \in \mathcal{R}_{L_I}$ in the Zariski closure of its orbit. Notice that, since $\widehat{\mathcal{R}}_{{P}_I}^\star$ is orbitwise-closed, we also have $\rho' \in \widehat{\mathcal{R}}_{{P}_I}^\star$.

    Now, suppose that this representation $\rho'$ is not $L_I$-irreducible. This means that there exists a proper parabolic subgroup $P'$ of $L_I$ containing the image of $\rho'$. But this implies that, for some $\ell \in L_I$, $\ell P' \ell^{-1}$ is a standard parabolic subgroup of $L_I$ and thus corresponds to a subset $I' \subseteq I$, since $I$ is the set of simple roots of $L_I$ by Remark \ref{rem:root-datum-levi}. In particular, the image of $\ell\rho'\ell^{-1}$ is contained in the standard parabolic subgroup $P_{I'} \supseteq P'$ of $G$ corresponding to $I' \subseteq \Delta$. But this contradicts the fact that $\ell\rho'\ell^{-1}$, and thus $\rho'$, lie in $\widehat{\mathcal{R}}_{{P}_I}^\star$.
\end{proof}

\begin{prop}\label{prop:polystability}
Every element of $\mathcal{R}_{L_I}^\star \subseteq \mathcal{R}_G$ is polystable for the adjoint action of $G$.
\end{prop}

\begin{proof}
    We will use the Hilbert-Mumford criterion for polystability, as can be found in \cite[Theorem 3.3]{saiz2021geometric} and \cite{newstead1978introduction}. In this case, it states that $\rho \in \mathcal{R}_G$ is polystable if and only if, for any $1$-parametric subgroup $\lambda: \GG_m \to G$ such that $\lim_{t\to 0} \lambda(t)\rho\lambda(t)^{-1}$ exists, there exists an element $g \in G$ such that $\lambda(t) g\rho g^{-1} \lambda(t)^{-1} = g\rho g^{-1}$ for all $t$.

    In this way, take $\rho \in \mathcal{R}_{L_I}^\star$ with $I \subseteq \Delta$ fixed, and suppose that $\lambda$ is a $1$-parametric subgroup satisfying that $\lim_{t\to 0} \lambda(t)\rho\lambda(t)^{-1}$ exists. Since the image of $\lambda$ is an abelian subgroup of $G$, it belongs to a maximal torus and, since all the maximal tori are conjugate, there exists $g \in G$ such that $g^{-1}\lambda g: \GG_m \to T$ is a cocharacter, where $T$ is the fixed maximal torus. In particular, $g^{-1}\lambda g = \lambda_J$ is the cocharacter associated to a certain subset $J \subseteq \Delta$. Notice that since $\lim_{t\to 0} \lambda(t)\rho\lambda(t)^{-1}$ exists, then also $\lim_{t\to 0} \lambda_J(t)\rho\lambda_J(t)^{-1}$ exists.
    
    Now, recall that $\rho(\Gamma) \subseteq P_I$, the parabolic subgroup of those $g \in G$ for which $\lim_{t\to 0} \lambda_I(t)g \lambda_I(t)^{-1}$ exists, where $\lambda_I$ is the cocharacter associated to $I$. This implies that $J \supseteq I$ or, equivalently, $T_J \subseteq T_I$ for the associated subtori. But $L_I = Z_G(T_I)$ and $\rho(\Gamma) \subseteq L_I$, so in particular $\lambda_J(t) \rho \lambda_J(t)^{-1} = \rho$. Unraveling the definition of $\lambda_J$, this implies that $\lambda(t) g\rho g^{-1} \lambda(t)^{-1} = g\rho g^{-1}$, as we wanted to prove.
\end{proof}

\begin{rem}
    As a corollary of the upcoming Theorem \ref{thm:core-key}, it will turn out that the polystable points of $\mathcal{R}_G$ are exactly the orbits of points of $\mathcal{R}_{L_I}^\star$ for some $I\subseteq \Delta$.
\end{rem}

\begin{defn} Let $L$ be a Levi subgroup. A closed subgroup $H \subseteq L$ is said to be \emph{sufficiently representative} of $L$ if, for any Levi subgroup $L' \subseteq L$ with $H \subseteq L'$, we must have that $L' = L$.
\end{defn}

\begin{prop}\label{prop:normalizer-conjugation}
Let $L$ be a Levi subgroup with associated torus $T$ and suppose that $H \subseteq L$ is a sufficiently representative subgroup of $L$. If $g \in G$ satisfies that $gHg^{-1} \subseteq L$, then there exists $g_0 \in N_G(T)$ such that $g_0hg_0^{-1} = ghg^{-1}$ for all $h \in H$.  
\end{prop}

\begin{proof}
    Recall that $L = Z_G(T)$, for $T$ the torus associated to $L$. Consider $\tilde{H} := Z_G(H)$ the centralizer of $H$. Since $T$ is the connected component of the identity of $Z_G(L)$, we have $T \subseteq Z_G(L)$, which jointly with $Z_G(L) \subseteq Z_G(H) = \tilde{H}$ implies that $T \subseteq \tilde{H}$. In fact, $T$ is a maximal torus of $\tilde{H}$ since, otherwise, a larger torus $T \subseteq T' \subseteq \tilde{H}$ would satisfy
    $$
        Z_G(T) = L \supseteq Z_G(T') \supseteq Z_G(\tilde{H}) = Z_G(Z_G(H)) \supseteq H.
    $$
    Hence, since $H$ is a sufficiently representative subgroup of $L$, then the Levi subgroup $Z_G(T')$ coincides with $L = Z_G(T)$, implying that $T'=T$.

    Additionally, since $gHg^{-1} \subseteq L$, then $H \subseteq g^{-1}Lg$ and thus
    $$
        \tilde{H} = Z_G(H) \supseteq Z_G(g^{-1}Lg) = g^{-1}Z_G(L)g \supseteq g^{-1}Tg.
    $$
    Therefore, $g^{-1}Tg$ is another maximal torus of $\tilde{H}$. Since any two maximal tori are conjugate, there exists $\tilde{h} \in \tilde{H}$ such that $\tilde{h}g^{-1} T g\tilde{h}^{-1}=T$.
    
    In this situation, our desired element is $g_0 := g\tilde{h}^{-1} \in N_G(T)$, which satisfies that, for any $h \in H$, we have $g_0hg_0^{-1} = g\tilde{h}^{-1}h\tilde{h}g^{-1} = ghg^{-1}$, since $\tilde{h} \in \tilde{H} = Z_G(H)$.
\end{proof}

\begin{rem}
    The key ideas for the proof of this result appear in a response by Sean Cotner in the MathOverflow forum \cite{SeanCMO}. Notice that $N_G(T)=N_G(L)$, so Proposition \ref{prop:normalizer-conjugation} also shows indeed the existence of an element $g_0 \in N_G(L)$ such that $g_0hg_0^{-1} = ghg^{-1}$ for all $h \in H$, in agreement with the aforementioned reference. We keep the proof of Proposition \ref{prop:normalizer-conjugation} in our manuscript for completeness.
\end{rem}

Indeed, by slightly adapting the previous proof, we can obtain a related result.

\begin{prop}\label{prop:conjugate-reps}
    Let $L$ and $L'$ be Levi subgroups and suppose that $H \subseteq L$ is a sufficiently representative subgroup of $L$. If, for some $g \in G$, we have $gHg^{-1} \subseteq L'$ and $gHg^{-1}$ is a sufficiently representative subgroup of $L'$, then $L$ and $L'$ are conjugate. 
\end{prop}

\begin{proof}
    Let $T$ and $T'$ be the tori associated to $L$ and $L'$, respectively. As in the proof of Proposition \ref{prop:normalizer-conjugation} take $\tilde{H} := Z_G(H)$. By the same argument as above, we have that $T \subseteq \tilde{H}$ is a maximal torus. Furthermore, since $gHg^{-1} \subseteq L' = Z_G(T')$, we also have that $g^{-1}T'g \subseteq \tilde{H}$ is a maximal torus. Hence, both $g^{-1}T'g$ and $T$ are conjugate (by an element of $\tilde{H} = Z_G(H)$), implying that also $T'$ and $T$ are so. Given that conjugation commutes with centralizers, we obtain that $L'$ and $L$ are conjugate.
\end{proof}

\begin{cor}\label{cor:core-subgroup}
Consider $\rho \in \mathcal{R}_{L_I}^\star$. If $g \rho g^{-1} \in \mathcal{R}_{L_I}^\star$ for a certain $g \in G$, then $g \rho g^{-1} = g_0 \rho g_0^{-1}$ for some $g_0 \in N_I$.
\end{cor}

\begin{proof}
Let $\gamma_1, \ldots, \gamma_s \in \Gamma$ be a set of generators of $\Gamma$ and take $H$ to be the Zariski closure of the subgroup generated by their images $\rho(\gamma_1), \ldots, \rho(\gamma_s) \subseteq L_I$. Since $\rho \in \mathcal{R}_{L_I}^\star$, we have that $H$ is a sufficiently representative subgroup of $L_I$. Since the element $g$ satisfies that $g\rho(\gamma_i)g^{-1} \in L_I$ for all $i$, then $gHg^{-1} \subseteq L_I$ and the result follows from Proposition \ref{prop:normalizer-conjugation}.
\end{proof}

\begin{lem}\label{lem:intersection-G-orbits}
Suppose that $\rho_1, \rho_2 \in \mathcal{R}_{L_I}^\star$ satisfy that $\overline{G \cdot \rho_1} \cap \overline{G \cdot \rho_2} \neq \emptyset$. Then, there exists $g_0 \in N_I$ such that $g_0 \cdot \rho_1 = \rho_2$.
\end{lem}

\begin{proof}
By Proposition \ref{prop:polystability}, the points of $\mathcal{R}_{L_I}^\star$ are polystable, and thus their orbits are closed. In particular, this means that any $\rho \in \overline{G \cdot \rho_1} \cap \overline{G \cdot \rho_2} = {G \cdot \rho_1} \cap {G \cdot \rho_2}$ satisfies $\rho = g_1 \rho_1 g_1^{-1} = g_2 \rho_2 g_2^{-1}$ for certain $g_1, g_2 \in G$. Hence, the existence of these elements implies that $g_2^{-1}g_1 \rho_1 g_1^{-1}g_2 = \rho_2$, so by Corollary \ref{cor:core-subgroup}, there exists $g_0 \in N_I$ such that $g_0 \cdot \rho_1 = g_0\rho_1 g_0^{-1} = \rho_2$, as desired.
\end{proof}

\begin{thm}\label{thm:core-key}
The pair $(\mathcal{R}_{L_I}^\star, N_I)$ is a core for the action of $G$ on $\widehat{\mathcal{R}}_{P_I}^\star$.
\end{thm}

\begin{proof}
Notice that, as $N_I = N_G(L_I)$, the subvariety $\mathcal{R}_{L_I}^\star$ is closed for the action of $N_I$. We will check that it fulfils the requirements of Definition \ref{defn:pseudo-quotient}. For part (i), observe that $\mathcal{R}_{L_I}^\star$ is polystable and $N_I$-invariant, so it is automatically orbitwise-closed. Part (ii) follows immediately from Corollary \ref{cor:intersection-core}.
For part (iii), let $C_1, C_2 \subseteq \mathcal{R}_{L_I}^\star$ be two disjoint $N_I$-invariants closed sets, and suppose that $\rho \in \overline{G \cdot C_1} \cap \overline{G \cdot C_2}$. Since $\widehat{\mathcal{R}}_{P_I}^\star$ is orbitwise-closed, we have that $\rho \in \widehat{\mathcal{R}}_{P_I}^\star$.
However, by Corollary \ref{cor:intersection-core}, we have that $\overline{G \cdot \rho} \cap \mathcal{R}_{L_I}^\star \neq \emptyset$, let say $\rho' \in \overline{G \cdot \rho} \cap \mathcal{R}_{I}^\star$. This means that the closures of the $G$-orbits of the three points $\rho_1, \rho_2$ and $\rho'$ intersect, so the three of them are related by a $N_I$-action by Lemma \ref{lem:intersection-G-orbits}. But this is impossible since $C_1$ and $C_2$ are disjoint and $N_I$-invariant.
\end{proof}

\begin{cor}\label{cor:virtual-core}
For any $I \subseteq \Delta$, we have the equality of virtual classes of $\KVar$
$$
    \left[\widehat{\mathcal{R}}_{P_I}^\star(\Gamma) \sslash G\right]= \left[\mathcal{R}_{L_I}^\star(\Gamma) \sslash N_I\right].
$$
\end{cor}

\begin{proof}
    By Corollary \ref{cor:equality-pseudo-quotients}, the virtual classes coincide for any pseudo-quotient of $\widehat{\mathcal{R}}_{P_I}^\star(\Gamma)$ with respect to the $G$-action and for any pseudo-quotient of $\mathcal{R}_{L_I}^\star(\Gamma)$ for the $N_I$-action. In particular, the usual GIT-quotients $\widehat{\mathcal{R}}_{P_I}^\star(\Gamma) \to \widehat{\mathcal{R}}_{P_I}^\star(\Gamma) \sslash G$ and $\mathcal{R}_{L_I}^\star(\Gamma) \to \mathcal{R}_{L_I}^\star(\Gamma) \sslash N_I$ are pseudo-quotients and thus their virtual classes agree.
\end{proof}

The Weyl group $W$ acts on the set $\Phi$ of roots. Denote by $2^{\Delta}$ the collection of subsets $I \subseteq \Delta$, where we have the following equivalence relation $\sim_W$: Given $I = \{\alpha_{i_1}, \ldots, \alpha_{i_s}\} \subseteq \Delta$ and $I' \subseteq \Delta$, then $I \sim_W I'$ if and only if there exists $\sigma \in W$ such that $\sigma \cdot \Phi_I = \Phi_{I'}$. Denote by $2^{\Delta}/\sim_W$ the quotient of $2^\Delta$ by this relation.

Notice that if $I \sim_W I'$, then the associated standard Levi subgroups $L_I$ and $L_{I'}$ are conjugate. Indeed, if $\sigma \cdot \Phi_I = \Phi_{I'}$, then $\sigma L_I \sigma^{-1} = L_{I'}$. In particular, ${\mathcal{R}}_{L_I}^\star(\Gamma)$ is isomorphic to $\mathcal{R}_{L_{I'}}^\star(\Gamma)$, being the isomorphism exactly conjugation by $\sigma$, and the isomorphism is equivariant for the respective actions of $N_I$ and $N_{I'}$. In particular, the virtual class $\left[\mathcal{R}_{L_I}^\star(\Gamma) \sslash N_I\right] \in \KVar$ is well defined for $I \in 2^{\Delta}/\sim_W$ by taking any representative.

\begin{thm}\label{thm:main-result}
For every reductive group $G$ and every finitely generated group $\Gamma$, we have that
$$
    \left[\mathcal{X}_G(\Gamma)\right] = \sum_{I \in 2^{\Delta}/\sim_W} \left[\mathcal{R}_{L_I}^\star(\Gamma) \sslash N_I\right],
$$
where $N_I$ is the normalizer of the Levi subgroup $L_I$ associated to $I$.
\end{thm}

\begin{proof}
    We start with the decomposition (\ref{eq:decomposition-representation}) of the representation variety $\mathcal{R}_{G}(\Gamma)$ into its parabolic parts as
$$
    \mathcal{R}_G(\Gamma) = \bigcup_{I \subseteq \Delta} \widehat{\mathcal{R}}_{P_I}^\star(\Gamma).
$$
Now, when we take the quotient by the action of $G$ by conjugation, Theorem \ref{thm:core-key} shows that $[\widehat{\mathcal{R}}_{P_I}^\star(\Gamma) \sslash G] = [\mathcal{R}_{L_I}^\star(\Gamma) \sslash N_I]$. 

However, these quotients are not disjoint, since we may have $\rho \in \mathcal{R}_{L_I}^\star(\Gamma)$ conjugate to $\rho' \in \mathcal{R}_{L_I'}^\star(\Gamma)$ for $I \neq I'$. But, by Proposition \ref{prop:conjugate-reps}, we have that in this case $L_I$ and $L_{I'}$ are also conjugate. Furthermore, by Lemma \ref{lem:conjugate-levi-weil}, $L_I$ and $L_{I'}$ must be conjugate by an element of the Weyl group $W$, and thus $I \sim_W I'$. Hence, by removing these redundancies and picking a unique element on each equivalence class of $2^{\Delta}/\sim_W$, we do get a decomposition of $\mathcal{R}_G(\Gamma) \sslash G$ into disjoint pieces of the form $\mathcal{R}_{L_I}^\star(\Gamma) \sslash N_I$, as required.
\end{proof}

The result in Theorem \ref{thm:main-result} can be slightly improved in the following setting. Notice that, for any $I \subseteq \Delta$ that defines a standard torus $T_I$, by definition we have a short exact sequence
\begin{equation}\label{eq:split-weyl}
    1 \longrightarrow L_I = Z_G(T_I) \longrightarrow N_I = N_G(T_I) \longrightarrow W_I = N_I / L_I \longrightarrow 1.
\end{equation}
In general, this sequence does not split, but in many cases it does (cf.\ \cite[Theorem 4.16]{adams2017lifting}), leading to a description of $N_I$ as semidirect product $N_I = L_I \rtimes W_I$.

In this situation, we can rewrite the quotient of Corollary \ref{cor:virtual-core} as
$$
\mathcal{R}_{L_I}^\star(\Gamma) \sslash N_I = \mathcal{R}_{L_I}^\star(\Gamma) \sslash (L_I \rtimes W_I) =  \left(\mathcal{R}_{L_I}^\star(\Gamma) \sslash L_I\right) \sslash W_I = \mathcal{X}_{L_I}^\star(\Gamma) \sslash W_I,
$$
where $\mathcal{X}_{L_I}^\star(\Gamma) \subseteq \mathcal{X}_{L_I}(\Gamma)$ is the open set of irreducible representations into $L_I$. Notice that the latter quotient is by $W_I$, which is a quotient by a finite group. Therefore, in this setting, we can improve Corollary \ref{cor:virtual-core} to get the following result.

\begin{cor}\label{cor:main-result-good-situation}
Take $I \subseteq \Delta$ and suppose that $N_I = L_I \rtimes W_I$, where $W_I = N_I/L_I$ is the Weyl group associated to $I$. Then, we have the equality of virtual classes of $\KVar$
$$
    \left[\widehat{\mathcal{R}}_{P_I}^\star(\Gamma) \sslash G\right]= \left[\mathcal{X}_{L_I}^\star(\Gamma) \sslash W_I\right].
$$
Moreover, provided that $N_I = L_I \rtimes W_I$ for any subset $I \subseteq \Delta$ of simple roots, we can describe the virtual class of the total character variety as
$$
    \left[\mathcal{X}_G(\Gamma)\right] = \sum_{I \in 2^{\Delta}/\sim_W} \left[\mathcal{X}_{L_I}^\star(\Gamma) \sslash W_I\right].
$$
\end{cor}

\begin{rem}
The character varieties $\mathcal{X}_{L_I}^\star(\Gamma)$ appearing in Corollary \ref{cor:main-result-good-situation} cannot be seen, in general, as geometric quotients $\mathcal{R}_{L_I}^\star(\Gamma) \sslash L_I$ of the corresponding representation varieties by the action of $L_I/Z(G)$. This issue runs deep, and its origin lies in the fact that not all irreducible representations are good although all good representations are irreducible (see \cite[Proposition 3.32]{florentino2012singularities}), as well as a generalization of Schur's lemma conjectured by Sikora and solved in \cite[Section 8]{guerin2022bad}. These bad representations classified in terms of BdS subgroups are very subtle and are associated with orbifold singularities. 
\end{rem}

We finish this section by discussing a special stratum in the decomposition of Theorem \ref{thm:main-result} and Corollary \ref{cor:main-result-good-situation}, namely, the one corresponding to $I = \emptyset$. In that case, as mentioned in Remark \ref{rmk:emptyset}, we have that $T_\emptyset = L_\emptyset = T$ is the maximal torus and $W_\emptyset = W$ is the standard Weyl group. Furthermore, $T$ has no proper parabolic subgroups, so in this case all representations are automatically irreducible. In this way, the stratum $I = \emptyset$ corresponds to the contribution
    $$
    \left[\mathcal{R}_{L_\emptyset}^\star(\Gamma) \sslash N_\emptyset\right]=\left[\mathcal{R}_{T}(\Gamma) \sslash N_G(T)\right].
    $$
    
    Now, for $I = \emptyset$, the sequence (\ref{eq:split-weyl}) reduces to
    $$
        1 \longrightarrow T = Z_G(T) \longrightarrow N_G(T) \longrightarrow W \longrightarrow 1.
    $$
    Even in this particular scenario, this sequence does not split in general \cite[Theorem 4.16]{adams2017lifting}. However, $T = Z_G(T)$ acts trivially on $\mathcal{R}_{T}(\Gamma)$, so we can still reduce the quotient to the action of $W = N_G(T)/T$, and thus
    $$
        \left[\mathcal{R}_{T}(\Gamma) \sslash N_G(T)\right]=\left[\mathcal{R}_{T}(\Gamma) \sslash W\right].
    $$
    Notice, in particular, that $\mathcal{X}_{T}(\Gamma) = \mathcal{R}_{T}(\Gamma)$.

    This is particularly interesting for $\Gamma = \ZZ^r$, since then $\mathcal{R}_{T}(\ZZ^r) = T^r$ and, therefore, the stratum $I = \emptyset$ is exactly
    $$
        \left[\mathcal{R}_{L_\emptyset}^\star(\ZZ^r) \sslash N_\emptyset\right]=\left[T^r \sslash W\right],
    $$
    in agreement with \cite[Theorem 2.1]{sikora2014character}. For certain groups $G$, the character variety $\mathcal{X}_G(\ZZ^r)$ is connected (e.g. $\SL_n$), and this is the only non-empty stratum in the decomposition of Theorem \ref{thm:main-result} (see \cite[Theorem 5.1]{florentino2014topology} for a full classification of these groups). However, in general, the character variety $\mathcal{X}_G(\ZZ^r)$ may have other non-trivial components corresponding to other classes $I \in 2^{\Delta} / \sim_W$, and the previous stratum only corresponds to the connected component of the identity representation (e.g. $\PGL_n$).
    
    Finally, we point out that the $W$-equivariant theory of $T^r$ has been fully described in \cite[Theorem 5.2]{florentino2021hodgeabelian}. Moreover, as shown in \cite[Theorem 4.4]{florentino2024mixed}, the Hodge structure of the whole connected component of the identity in the representation variety can be described by taking into account that it corresponds to
    $$
        \widehat{\mathcal{R}}_{B}(\ZZ^r) = \widehat{\mathcal{R}}_{T}(\ZZ^r) = G \cdot {\mathcal{R}}_{T}(\ZZ^r) = G \cdot T^r.
    $$
    The map $G/T \times_W T^r \to G \cdot T^r$ given by $[g, (t_1, \ldots, t_r)]_W \mapsto (gt_1g^{-1}, \ldots, gt_rg^{-1})$ is a regular bijective map between these algebraic varieties, showing that their virtual classes agree. The $W$-equivariant theory of this space, and also for general nilpotent groups, is studied in \cite[Corollary 4.6]{florentino2024mixed}.

\section{Parabolic stratification for classical groups}
\label{sec:parabolic-stratification-classical-groups}

In this section, we will discuss how the stratification developed in Section \ref{sec:parabolic-stratification} instances for several classical groups. First we study the most common and general case $\GL_n$, then particularize the results for the Langlands dual pair $\SL_n$ and $\PGL_n$ in type $A$. After that, we study the representatives of types $C$ and $B$, $\Sp_{2n}$ and $\SO_{2n+1}$, which are Langlands dual. Finally we cover the type $D$ case $\SO_{2n}$, which is Langlands self-dual. 

\subsection{Stratification for $G = \GL_n$}
\label{ssec:stratification_GLn}

In this case, the stratification of Theorem \ref{thm:main-result} captures the stratification of the representation variety by partition type, as constructed in \cite{FNZ2023generating} and outlined in the introduction. Recall that we are using the standard representation of $\GL_{n}$ in $\mathbb{C}^{n}$ and that a representation $\rho \in \mathcal{R}_{\GL_n}(\Gamma)$ is irreducible if there exists no non-trivial invariant subspace of $\CC^n$ or, equivalently, there is no proper parabolic subgroup where the image of $\rho$ is contained. Recall the notation of direct sum $\rho_{1}\oplus\rho_{2}\in\mathcal{R}_{\GL_{n_1+n_2}}(\Gamma)$ of representations $\rho_{1}:\Gamma\to \GL_{n_{1}}$ and $\rho_{2}:\Gamma\to \GL_{n_{2}}$.  

As explained in Section \ref{sec:lie-gl}, the simple positive roots of $\GL_n$ are $\Delta = \{\alpha_{12}, \ldots, \alpha_{(n-1)n}\}$, which we will simply denote by $\Delta = \{1, \ldots, n-1\}$, with $j$ corresponding to the root $\alpha_{j(j+1)}$. We choose as Borel subgroup $B$ the subgroup of upper triangular invertible matrices, which fixes a basis $e_1, e_2,\ldots, e_n$ of $\CC^n$. Then the standard parabolic subgroup $P_I$ associated to a subset $I \subseteq \Delta$ can be understood as follows. Suppose that the complement of $I$ is $\Delta\setminus I=\{i_1,i_2,\ldots,i_s\}$ with $i_1 < i_2 < \cdots < i_s$. Let us define the standard flag associated to $I$ by
\begin{equation}\label{eq:flag-parabolic}
    0 \subsetneq V_{1} \subsetneq V_{2} \subsetneq \cdots \subsetneq V_{s} \subsetneq \CC^n,
\end{equation}
where $V_j=\langle e_1,\ldots, e_{i_j} \rangle$. Note that the empty set $I=\emptyset$ corresponds to $\Delta\setminus I=\Delta=\{1,2,\ldots,n-1\}$ and to the full flag $0\subsetneq V_1\subsetneq V_2\subsetneq \cdots \subsetneq V_{n-1}\subsetneq \CC^n$, and that the whole $I=\Delta$ corresponds to $\Delta\setminus I=\emptyset$ and to the trivial flag $0\subsetneq \CC^n$. 

Then, the standard parabolic subgroup $P_I$ associated to $I \subseteq \Delta$ is the stabilizer of the flag (\ref{eq:flag-parabolic}), i.e.\ $P_I$ is the collection of $A \in \GL_n$ such that $A(V_j) \subseteq V_j$ for all $j = 1, \ldots, s$. In particular, the Borel subgroup $B$ of upper triangular matrices is the stabilizer of the full standard flag corresponding to $\emptyset\subseteq \Delta$, the whole $\GL_n$ is the stabilizer of the trivial flag $0\subsetneq \CC^n$ corresponding to the whole $\Delta$ and maximal standard parabolic subgroups coming from $I=\Delta\setminus\{i\}$ (removing one node $i$ in $\Delta$) correspond to one-step flags $0\subsetneq V_{1}\subsetneq \CC^n$. 

\begin{rem}
If we do  not fix a Borel $B\subseteq G$, subsets $I\subseteq \Delta$ correspond to conjugacy classes of parabolic subgroups $\mathcal{P}_I$. Each parabolic subgroup in the conjugacy class $\mathcal{P}$ stabilizes a flag of the same numerical invariants, i.e. stabilizes a flag of the form $0 \subsetneq V_{1} \subsetneq V_{2} \subsetneq \cdots \subsetneq V_{s} \subsetneq \CC^n$ where the dimensions of the terms are $\dim V_{j}=i_j$ and $I=\Delta\setminus \{i_1,\ldots, i_s\}$. 
\end{rem}

Furthermore, once fixed the flag (\ref{eq:flag-parabolic}) corresponding to a subset $I \subseteq \Delta$, the Levi subgroup $L_I$ associated to the standard parabolic subgroup $P_I$ can be obtained as follows. Consider the quotients of the flag (\ref{eq:flag-parabolic}) given by
$$
    W_j =  \langle e_{i_{j-1}+1}, \ldots, e_{i_j}\rangle=V_j / V_{j-1}
$$
for $j = 1, \ldots, s + 1$, where we set $V_0 = 0$ and $V_{s+1}=\CC^n$. Then, we have a splitting
$$
   \CC^n = \bigoplus_{j = 1}^{s+1} W_j, \qquad  V_k = \bigoplus_{j \leq k} W_j.
$$
The Levi subgroup $L_I$ is then the collection of $A\in \GL_n$ that preserve this splitting, i.e.\ such that $A(W_j) \subseteq W_j$ for all $1 \leq j \leq s+1$. In particular, setting $i_0=0$ and $i_{s+1} = n$, we have $\dim W_j = i_j - i_{j-1} =: \lambda_j$ and thus
\begin{equation}\label{eq:levi-gln}
    L_I \cong \prod_{j = 1}^{s+1} \GL_{\lambda_j}.
\end{equation}

\begin{example}\label{ex:GL_n}
Let us see this in the concrete case of $G=\GL_5$. The roots in this case are \linebreak $\Phi=\{\alpha_{ij}:1\leq i,j\leq 5, i\neq j\}$, the positive roots are $\Phi^+=\{\alpha_{ij}:1\leq i< j\leq 5\}$ and the simple roots are $\Delta = \{\alpha_{12} = 1, \alpha_{23} = 2, \alpha_{34} = 3, \alpha_{45} = 4\}$. Let us consider the subset $I=\{\alpha_{23} = 2, \alpha_{45} = 4\}\subseteq \Delta$. The corresponding kernels are
\[\ker \alpha_{23}=\left\{{\rm diag}(a_1, b, b, a_4, a_5)\right\}\; , \quad \ker \alpha_{45}=\left\{{\rm diag}(a_1, a_2, a_3, c, c)\right\}.\]
Therefore 
\[T_I= \bigcap_{\alpha\in I} \ker \alpha=\left\{{\rm diag}(a, b, b, c, c):a,b,c\in \GG_m\right\}.\]

Since $\Delta \setminus I = \{1,3\}$, the parabolic subgroup $P_I$ is the stabilizer of the flag
$$
    0 \subsetneq V_1 = \langle e_1\rangle \subsetneq V_2 = \langle e_1, e_2, e_3 \rangle \subsetneq \CC^5,
$$
which corresponds exactly to matrices of the form
$$
    \left(\begin{array}{c|cc|cc}
        \ast & \ast & \ast & \ast & \ast \\\cline{1-1}
         & \ast & \ast & \ast & \ast \\
         & \ast & \ast & \ast & \ast \\\cline{2-3}
        \multicolumn{1}{r}{}  &  &  & \ast & \ast \\
        \multicolumn{1}{r}{}  &  &  & \ast & \ast \\\cline{4-5}
    \end{array}\right).
$$

The Levi subgroup is $L_I=Z_G(T_I)$. It corresponds to those matrices stabilizing the splitting
$$
    W_1 = \langle e_1\rangle, \quad W_2 = \langle e_2, e_3\rangle, \quad W_3 = \langle e_4, e_5\rangle.
$$
We have thus that $L_I \cong \GL_1 \times \GL_2 \times \GL_2$. Explicitly, it is given by block-diagonal matrices of the form
$$
    \left(\begin{array}{c|cc|cc}
        \ast & \multicolumn{4}{r}{} \\\cline{1-3}
         & \ast & \ast & \multicolumn{2}{r}{} \\
         & \ast & \ast & \multicolumn{2}{r}{} \\\cline{2-5}
        \multicolumn{1}{r}{}  &  &  & \ast & \ast \\
        \multicolumn{1}{r}{}  &  &  & \ast & \ast \\\cline{4-5}
    \end{array}\right).
$$
The roots $\Phi_I$ of $L_I$ are $\alpha_{23},\alpha_{32},\alpha_{45},\alpha_{54}$, and the positive ones come from $\Phi^{+}\cap \Phi_I = \{\alpha_{23}, \alpha_{45}\}$.
The cocharacter $\lambda_I$ associated to the subset $I$ is computed such that $(\alpha_{23}, \lambda_I)=(\alpha_{45}, \lambda_I)=0$. Therefore, it is given by 
\[\lambda_I:\GG_m \rightarrow T\;, \quad t\mapsto {\rm diag}(t^{\lambda},t^{\mu}, t^{\mu}, t^{\nu}, t^{\nu})\;, \lambda, \mu, \nu\in\mathbb{Z}.\]
Finally, the normalizer $N_{\GL_5}(T_I) = L_I \rtimes S_2$, where the symmetric group $S_2$ acts on $L_I$ by permuting the two blocks of rank two. Hence, the Weyl group is $W_I = N_{\GL_5}(T_I) / Z_{\GL_5}(T_I) = (L_I \rtimes S_2) / L_I = S_2$.

\end{example}

Notice that the subsets of $\Delta$ can be identified with tuples $\bm{\lambda} = (\lambda_1, \ldots, \lambda_r)$ of positive integers with $\sum_i \lambda_i = n$, i.e.\ ordered partitions of $n$. Explicitly, given $I \subseteq \Delta$, let $\Delta \setminus I = \{i_1, \ldots, i_s\}$ with \linebreak $0< i_1 < i_2 < \cdots < i_s<n$ and set $i_0 = 0$ and $i_{s+1} = n$. Then the associated ordered partition is $\bm{\lambda} = (\lambda_1, \ldots, \lambda_{s+1})$ with $\lambda_j = i_j - i_{j-1}$ for $1 \leq j \leq s+1$, and the process is reversible. We shall denote the subset of $\Delta$ associated to an ordered partition $\bm{\lambda}$ by $I_{\bm{\lambda}} \subseteq \Delta$.

Since the Weyl group of $\GL_n$ is the symmetric group $W = S_n$, we have that two subsets $I = \{\alpha_{\ell_1(\ell_1+1)}, \ldots, \alpha_{\ell_s(\ell_s+1)}\}$ and $I' = \{\alpha_{\ell_1'(\ell_1'+1)}, \ldots, \alpha_{\ell_s'(\ell_s'+1)}\}$ of $\Delta$ are equivalent under the action of $W$ if and only if the cycle structure, as a product of transpositions, of the permutation $(\ell_1\, \ell_1+1)\cdots (\ell_s\, \ell_s+1) \in S_n$ is the same as the one of $(\ell_1'\, \ell_1'+1)\cdots (\ell_s'\, \ell_s'+1) \in S_n$ or, equivalently, if they are conjugate permutations of $S_n$. A convenient way of capturing these equivalence classes is through unordered partitions (as in \cite{FNZ2023generating}) $[k]=[1^{k_{1}}\cdots j^{k_{j}}\cdots n^{k_{n}}]$ where
$k_{j}$ means that $[k]$ has $k_{j}\geq 0$ parts of
size $j\in\{1,\ldots,n\}$, $n=\sum_{j=1}^{n}j\cdot k_{j}$ and whose \emph{length} is the sum of the exponents $|[k]|:=\sum k_{j}$. We denote by $\mathcal{P}_{n}$ the set of partitions
of $n\in\mathbb{N}$.

With this notion at hand, we have a natural bijection between unordered partitions of $n$ and equivalence classes of $ 2^\Delta / \sim_W$. Indeed, given $[k] \in \mathcal{P}_n$, it defines a canonical ordered partition just by sorting the entries of $[k]$ increasingly, that we will also denote by $[k]$. This gives rise to a well-defined subset $I_{[k]} \subseteq \Delta$ giving a canonical representative of its $\sim_W$-equivalence class.

\begin{example}
    Consider $G = \GL_5$, so that the simple roots are $\Delta = \{1, 2,3,4\}$. Let us take the unordered partition $[k] = [1^12^2] \in \mathcal{P}_5$, which gives rise to the ordered partition $(1,2,2)$ of $5$. In this manner, $\Delta \setminus I_{[k]} = \{1,3\}$ and therefore $I_{[k]} = \{2,4\}$. The standard parabolic subgroup $P_{[k]}$ associated to the partition $[k]$ is thus exactly the one studied in Example \ref{ex:GL_n}. 
\end{example}

In the same vein, following (\ref{eq:levi-gln}) the associated Levi subgroup $L_{[k]} = L_{I_{[k]}}$ is isomorphic to
$$
    L_{[k]} = \prod_{j=1}^n \GL_{j}^{k_j}.
$$
Furthermore, if $S_{[k]} = S_{k_1} \times S_{k_2} \times \cdots \times S_{k_n}$ denotes the subgroup of $S_n$ permuting blocks of the same size, then we have that $N_{[k]} = N_{I_{[k]}} = L_{[k]} \rtimes S_{[k]}$, where $S_{[k]}$ acts on $L_{[k]}$ by permutation of blocks.

Hence, $\rho \in \mathcal{R}_{L_{[k]}}^\star(\Gamma)$ if and only if $\rho$ has the form
\begin{equation*}
    \rho = \bigoplus_{j=1}^{n}\bigoplus_{\ell = 1}^{k_j}\rho_{j,\ell},
\end{equation*}
where each $\rho_{j,\ell} \in \mathcal{R}^{\star}_{\GL_{j}}(\Gamma)$ is irreducible, that is, $\rho$ is a direct sum of $k_{j}$ irreducible representations of rank $j$, for $j=1,\ldots,n$. By convention, if some $k_{j}=0$, then $\rho_{j, \ell}$ is not present in the direct sum. In other words, we have that
$$
\mathcal{R}^{\star}_{L_{[k]}}(\Gamma) = \prod_{j=1}^n \mathcal{R}^{\star}_{\GL_j}(\Gamma)^{k_j}, \qquad \mathcal{X}^{\star}_{L_{[k]}}(\Gamma) = \prod_{j=1}^n \mathcal{X}^{\star}_{\GL_j}(\Gamma)^{k_j}.
$$

Therefore, we are in the situation of Corollary \ref{cor:main-result-good-situation} and thus we get that
\begin{align*}
    \left[\widehat{\mathcal{R}}_{P_{I_{[k]}}}^\star(\Gamma) \sslash G\right] & =  \left[\mathcal{X}^{\star}_{L_{[k]}}(\Gamma) \sslash S_{[k]}\right] = \left[\left(\prod_{j=1}^n \mathcal{X}^{\star}_{\GL_j}(\Gamma)^{k_j}\right) \sslash S_{[k]}\right],
\end{align*}
Notice that these character varieties are actually geometric quotients, since the action of $\PGL_j$ on $\mathcal{R}^{\star}_{\GL_j}(\Gamma)$ is free by the Schur's lemma (cf.\ \cite[Lemma 1.7]{fulton2013representation}). Now, observe that we have a natural identification
$$
\left(\prod_{j=1}^n \mathcal{X}^{\star}_{\GL_j}(\Gamma)^{k_j}\right) \sslash S_{[k]}  = \prod_{j=1}^{n}\sym^{k_{j}}(\mathcal{X}^{\star}_{\GL_j}(\Gamma)),
$$
where $\sym^{k}(X) = X^k/S_k$ denotes the symmetric product of the variety $X$ (cf.\ \cite[Proposition 4.5]{FNZ2023generating}). Hence, denoting this latter space by $\mathcal{X}^{[k]}_{\GL_n}(\Gamma)$, Theorem \ref{thm:main-result} recovers the following result previously proven in \cite{FNZ2023generating}.

\begin{prop}\cite[Proposition 4.3]{FNZ2023generating}
The character variety $\mathcal{X}_{\GL_n}(\Gamma)$ can be written as a
disjoint union, labelled by partitions $[k]\in\mathcal{P}_{n}$, of
locally closed quasi-projective varieties as
\[
\mathcal{X}_{\GL_n}(\Gamma)=\bigsqcup_{[k]\in\mathcal{P}_{n}}\mathcal{X}^{[k]}_{\GL_n}(\Gamma).
\]
\end{prop}

\subsection{Stratification for $G = \SL_n$ and $\PGL_n$ (Dynkin diagram $A_{n-1}$)}\label{sec:SL-PGL}
These cases are very similar to the one of $\GL_n$, as expected from the fact that their simple roots form the same root system $A_{n-1}$. Langlands duality of the groups will be reflected in a symmetry in the decomposition.

As before, in both cases the collection of Weyl-equivalence classes of subsets of simple positive roots is in bijection with partitions $[k] = [1^{k_{1}}\cdots j^{k_{j}}\cdots n^{k_{n}}] \in \mathcal{P}_n$ of $n$. 
The unique subtlety is that, now, all the elements must have determinant $1$ for $\SL_n$ and are determined up to re-scaling for $\PGL_n$. In this way, the Levi subgroups $L_{[k]}^{\SL_n}$ and $L_{[k]}^{\PGL_n}$ of $\SL_n$ and $\PGL_n$ respectively are now
\begin{align*}
    L_{[k]}^{\SL_n} & = \left\{(A_{j,\ell}) \in \prod_{j=1}^n \GL_{j}^{k_j} \,\left|\, \prod_{j,\ell} \det(A_{j,\ell}) = 1\right.\right\} \subseteq \prod_{j=1}^n \GL_{j}^{k_j},\\
    L_{[k]}^{\PGL_n} & = \left\{(A_{j,\ell}) \in \prod_{j=1}^n \GL_{j}^{k_j}\right\}/\GG_m,
\end{align*}
where $\GG_m$ acts by simultaneous re-scaling of the matrices. The normalizer of the Levi subgroup has the same form $N_{[k]}^{\SL_n} = L_{[k]}^{\SL_n} \rtimes S_{[k]}$ (resp.\ $N_{[k]}^{\PGL_n} = L_{[k]}^{\PGL_n} \rtimes S_{[k]}$), where $S_{[k]}$ acts on the Levi subgroup by permutation of blocks.

\begin{rem}
    These Levi subgroups agree with the fact that the group $L_{[k]}^{\SL_n}$ (resp.\ $L_{[k]}^{\PGL_n}$) is the unique reductive group with root datum given by $(\ZZ^n/(1,\ldots,1), \Phi_{I_{[k]}}, (\ZZ^n/(1,\ldots,1))^\vee, \Phi_{I_{[k]}}^\vee)$ (resp.\ with root datum $((\ZZ^n/(1,\ldots,1))^\vee, \Phi_{I_{[k]}}^\vee, \ZZ^n/(1,\ldots,1), \Phi_{I_{[k]}})$), as claimed by Remark \ref{rem:root-datum-levi}.
\end{rem}

In this way, a representation $\rho \in \mathcal{R}_{L_{[k]}^{\SL_n}}^\star(\Gamma) $ (resp.\ $\rho \in \mathcal{R}_{L_{[k]}^{\PGL_n}}^\star(\Gamma) $) if and only if $\rho$ has the form
\begin{equation*}
    \rho = \bigoplus_{j=1}^{n}\bigoplus_{\ell = 1}^{k_j}\rho_{j,\ell}, \qquad \left(\textrm{resp.\ } \rho = \left(\bigoplus_{j=1}^{n}\bigoplus_{\ell = 1}^{k_j}\rho_{j,\ell}\right)/\GG_m \right)
\end{equation*}
where each $\rho_{j,\ell} \in \mathcal{R}^{\star}_{\GL_{j}}(\Gamma)$ is irreducible. In the case $G = \SL_n$, we must have $\prod_{j,\ell}\det(\rho_{j, \ell}) = 1$, and in the case $G = \PGL_n$ we have a re-scaling action of $\GG_m$. Therefore, we have
\begin{align*}
    \mathcal{X}_{L_{[k]}^{\SL_n}}^\star(\Gamma) & \sslash S_{[k]} = \left\{(\rho_{j,\ell}) \in \prod_{j=1}^n \mathcal{X}^{\star}_{\GL_j}(\Gamma)^{k_j} \,\left|\, \prod_{j,\ell} \det(\rho_{j,\ell}) = 1\right.\right\} \sslash S_{[k]},\\
    \mathcal{X}_{L_{[k]}^{\PGL_n}}^\star(\Gamma) & \sslash S_{[k]} = \left(\left(\prod_{j=1}^n \mathcal{X}^{\star}_{\GL_j}(\Gamma)^{k_j}\right) / \GG_m \right) \sslash S_{[k]}.
\end{align*}
Notice that the former is naturally the subset of representations of $\prod_{j=1}^{n}\sym^{k_{j}}(\mathcal{X}^{\star}_{\GL_j}(\Gamma))$ with determinant $1$. Denoting these subsets by $\mathcal{X}^{[k]}_{\SL_n}(\Gamma)$ and $\mathcal{X}^{[k]}_{\PGL_n}(\Gamma)$ respectively, we get the analogous decompositions
$$
\mathcal{X}_{\SL_n}(\Gamma)=\bigsqcup_{[k]\in\mathcal{P}_{n}}\mathcal{X}^{[k]}_{\SL_n}(\Gamma), \qquad \mathcal{X}_{\PGL_n}(\Gamma)=\bigsqcup_{[k]\in\mathcal{P}_{n}}\mathcal{X}^{[k]}_{\PGL_n}(\Gamma).
$$

\begin{rem}
    The previous decomposition shows that the expected equality of $E$-polynomials predicted by the Langlands duality 
    $$
    e\left(\mathcal{X}_{\SL_n}(\Gamma)\right) = e\left(\mathcal{X}_{\PGL_n}(\Gamma)\right)
    $$
    holds provided that it also holds the equality stratum by stratum
    $$
    e\left(\mathcal{X}^{[k]}_{\SL_n}(\Gamma)\right) = e\left(\mathcal{X}^{[k]}_{\PGL_n}(\Gamma)\right)
    $$
    for all $[k] \in \mathcal{P}_n$. This equality was proven in \cite{florentino2021serre} by using the stratification and equalities stratum by stratum, for the $\SL_n$ and $\PGL_n$-character varieties of the free group. 
\end{rem}

\subsection{Stratification for $G = \Sp_{2n}$ (Dynkin diagram $C_n$)}

In this case, with the notation of Section \ref{sec:lie-sp}, the simple roots are $\Delta=\{\beta_{12,+ -}, \beta_{23,+-}, \beta_{34,+-},\ldots, \beta_{(n-1) n, +-}, \alpha_{n, +}\}$, the last one being the unique long root in the Dynkin diagram. We index them as $\Delta = \{1,\ldots, n\}$ with $j$ corresponding to $\beta_{j(j+1),+-}$ for $1 \leq j \leq n-1$ and $n$ corresponding to $\alpha_{n,+}$.

If $\omega$ is the symplectic form in $\CC^{2n}$, given a subspace $W\subseteq \CC^{2n}$, define its $\omega$-orthogonal as \linebreak $W^{\perp}=\{v\in \CC^{2n} \,\mid \omega(w,v)=0, \; \forall w\in W\}$. A subspace $W\subseteq \CC^{2n}$ is called isotropic if $W\subseteq W^{\perp}$, i.e. if $\omega|_{W\times W}=0$. An isotropic subspace has dimension at most $n$ and is contained in a maximal isotropic subspace. These maximal isotropic subspaces $W$, known as lagrangian subspaces, are exactly the $n$-dimensional isotropic subspaces and satisfy $W^\perp = W$.

Given a subset $I \subseteq \Delta$, let $\Delta \setminus I = \{i_1, \ldots, i_s\}$ with $i_1 < i_2 < \cdots < i_s$. We choose a basis of $\CC^{2n}$ of the form $\langle x_1, \ldots, x_n, y_1, \ldots, y_n\rangle$ for which the symplectic form is the standard one. In this situation, the standard parabolic subgroup $P_I$ associated to $I$ is the subgroup of symplectic matrices stabilizing the flag
$$
    0 \subsetneq V_{1}  \subsetneq \cdots \subsetneq V_{j} \subsetneq \cdots \subsetneq V_{s} \subsetneq \CC^{2n}.
$$
where $V_{j} = \langle x_{1}, \ldots, x_{i_j}\rangle$ for $1 \leq j \leq s$ are isotropic subspaces. Notice that, since a symplectic map preserving an isotropic space $V$ also preserves its orthogonal $V^\perp$, we get that $P_I$ is also the stabilizer of the flag
\begin{equation}\label{eq:flag-simplectic}
    0 \subsetneq V_{1} \subsetneq \cdots \subsetneq V_{j} \subsetneq \cdots \subsetneq V_{s} \subsetneq V_{s}^\perp \subsetneq V_{{s-1}}^\perp \subsetneq \cdots \subsetneq V_{1}^\perp \subsetneq \CC^{2n}.
\end{equation}
where, explicitly, $V_j^\perp = \langle x_1, \ldots, x_n, y_{i_j+1}, \ldots, y_n\rangle$.

\begin{rem}
    The empty set $I=\emptyset$ corresponds to $\Delta\setminus I=\Delta=\{1,2,\ldots,n\}$ and to the maximal/full isotropic flag $0\subsetneq V_1\subsetneq V_2\subsetneq \cdots \subsetneq V_{n-1}\subsetneq V_{n}\subsetneq \CC^{2n}$, and the whole $I=\Delta$ corresponds to $\Delta\setminus I=\emptyset$ and to the trivial flag $0\subsetneq \CC^{2n}$. There is one special choice of node, the long root $n$, yielding the Siegel parabolics associated to $\Delta\setminus \{n\}$ which corresponds to the minimal flag $0\subsetneq V_1=\langle x_1,\ldots, x_n\rangle \subsetneq \CC^{2n}$, where note that $V_1$ is lagrangian (i.e.\ maximal isotropic).
Each other choice of node $1\leq i_1\leq n-1$ yields a maximal standard parabolic  associated to $\Delta\setminus \{i_1\}$ corresponding to the minimal flag $0\subsetneq V_{1}= \langle x_1,\ldots, x_{i_1} \rangle \subsetneq \CC^{2n}$, where note that $V_{1}$ is isotropic but not lagrangian. 
\end{rem}

Similarly, for the Levi subgroup $L_I$ associated to $I$, we have that $L_I$ is the collection of symplectic maps preserving the graded pieces of the flag (\ref{eq:flag-simplectic}). To describe it explicitly, we have two situations.
\begin{itemize}
    \item If $n \notin I$, then $\Delta \setminus I = \{i_1, \ldots, i_s = n\}$ with $i_1 < \cdots < i_s = n$. This means that the flag (\ref{eq:flag-simplectic}) contains the lagrangian subspace $V_{s} = \langle x_1, \ldots, x_n\rangle=V_{s}^{\perp}$. Now, if we set $i_0 = 0$ and $V_{0}=0$ we can consider
    $$
        W_j = \langle x_{i_{j-1}+1}, \ldots, x_{i_j}\rangle=V_j / V_{j-1}, \qquad W_j' = \langle y_{i_{j-1}+1}, \ldots, y_{i_j}\rangle=V_{j-1}^{\perp}/V_{j}^{\perp}, \qquad j=1,\ldots, s,
    $$
    which yield a splitting of the flag (\ref{eq:flag-simplectic}) in the sense that
    $$
        \CC^{2n}= \bigoplus_{j \leq s} W_j 
        \oplus \bigoplus_{j \leq s} W'_j,\qquad
        V_k=\bigoplus_{j \leq k} W_j, \qquad V_{k}^\perp = \bigoplus_{j \leq s} W_j \oplus \bigoplus_{j > k} W_j'=\langle x_1, \ldots, x_n\rangle\oplus \bigoplus_{j > k} W_j'.
    $$
    Then, $L_I$ is the subspace of symplectic maps preserving $W_j$ and $W_j'$ for all $1 \leq j \leq s$. However, since both $W_j$ and $W_j'$ are isotropic spaces, the symplectic form vanishes on them, so the only constraint is that if $A$ acts on $W_j$ as a linear endomorphism $A_j \in \GL_{\lambda_j}$, where $\lambda_j = i_{j} -i_{j-1}$, then $A$ acts on $W_j'$ as $(A_j^t)^{-1}$. Hence, we get
    \begin{equation}\label{eq:levi-sp-not-n}
        L_I \cong \prod_{j=1}^s \GL_{\lambda_j},
    \end{equation}
    where the isomorphism is given by $\prod_{j=1}^s \GL_{\lambda_j} \ni (A_1, \ldots, A_s) \mapsto (A_1, \ldots, A_s, (A_s^t)^{-1}, \ldots, (A_1^t)^{-1}) \in L_I$.
    \item If $n \in I$, then $\Delta \setminus I = \{i_1, \ldots, i_s\}$ with $i_s < n$. This means that the flag (\ref{eq:flag-simplectic}) does not contain the lagrangian subspace $\langle x_1, \ldots, x_n\rangle$. The splitting of (\ref{eq:flag-simplectic}) is now given again by
    $$
        W_j = \langle x_{i_{j-1}+1}, \ldots, x_{i_j}\rangle, \qquad W_j' = \langle y_{i_{j-1}+1}, \ldots, y_{i_j}\rangle. 
    $$
    for $1 \leq j \leq s$, but we must also add the new subspace
    $$
        W^\ast = \langle x_{i_s+1}, \ldots, x_n, y_{i_s+1}, \ldots, y_n\rangle
    $$
    such that 
    \[\CC^{2n}= \bigoplus_{j \leq s} W_j 
        \oplus \bigoplus_{j \leq s} W'_j\oplus W^{\star},\]
    Then, $L_I$ is the subspace of symplectic maps preserving $W_j$ and $W_j'$ for all $1 \leq j \leq s$, as well as $W^\ast$. Again, since $W_j$ and $W_j'$ are isotropic spaces, they impose no extra conditions. But in $W^\ast$ the situation is different since there the restriction is a genuine symplectic form. Hence, we have that in this case
    \begin{equation}\label{eq:levi-sp-n}
        L_I \cong \prod_{j=1}^s \GL_{\lambda_j} \times \Sp_{2(n-i_s)},
    \end{equation}
    where $\lambda_j = i_{j} -i_{j-1}$.
\end{itemize}

\begin{example}
Consider the case $n = 5$, so $\Delta = \{1, \ldots, 5\}$. Take $I = \{1,4\}$ so that $\Delta \setminus I = \{i_1 = 2, i_2 = 3, i_3 = 5\}$. Then, the flag (\ref{eq:flag-simplectic}) corresponds to
\begin{align*}
    0 \subsetneq V_{1} = \langle x_1, x_2\rangle & \subsetneq V_{2} = \langle x_1, x_2, x_3\rangle \subsetneq V_{3} = \langle x_1, \ldots, x_5\rangle = V_{3}^\perp \\
    &\subsetneq V_{2}^\perp = \langle x_1, \ldots, x_5, y_4, y_5\rangle \subsetneq V_{1}^\perp = \langle x_1, \ldots, x_5, y_3, y_4, y_5\rangle \subsetneq \CC^{10}.
\end{align*}
Then, the associated splitting is
\begin{align*}
    W_{1} & = \langle x_1, x_2\rangle, \quad W_{2} = \langle x_3\rangle, \quad W_{3} = \langle x_4, x_5\rangle \\
    W_{1}' & = \langle y_1, y_2\rangle, \quad W_{2}' = \langle y_3\rangle, \quad W_{3}' = \langle y_4, y_5\rangle.
\end{align*}
Therefore, taking into account that the action on the $W_1$, $W_2$ and $W_3$ determines the one on $W_1'$, $W_2'$ and $W_3'$, we have
$$
    L_I = \GL_2 \times \GL_1 \times \GL_2.
$$
\end{example}

\begin{example}
Now, consider again $n=5$ but the subset $I = \{1,4,5\}$ instead, so that $\Delta \setminus I = \{i_1 = 2, i_2 = 3\}$. Then, the flag (\ref{eq:flag-simplectic}) corresponds to
\begin{align*}
    0 \subsetneq V_{1} = \langle x_1, x_2\rangle \subsetneq V_{2} = \langle x_1, x_2, x_3\rangle \subsetneq V_{2}^\perp = \langle x_1, \ldots, x_5, y_4, y_5\rangle \subsetneq V_{1}^\perp = \langle x_1, \ldots, x_5, y_3, y_4, y_5\rangle \subsetneq \CC^{10}.
\end{align*}
Then, the associated splitting is
\begin{align*}
    W_{1} = \langle x_1, x_2\rangle, \quad W_{2} = \langle x_3\rangle, \quad W_{1}' = \langle y_1, y_2\rangle, \quad W_{2}' = \langle y_3\rangle,\quad W^\ast = \langle x_4, x_5, y_4, y_5\rangle.
\end{align*}
Therefore, we have
$$
    L_I = \GL_2 \times \GL_1 \times \Sp_{4}.
$$
\end{example}

\begin{rem}
    The distinction above into the cases $n \not\in I$ and $n \in I$ is very natural if we think on the fact that the Dynkin diagram of $L_I$ is exactly $I$ with the edges between vertices in $I$. Recall that the Dynkin diagram $C_n$ has a double edge between the vertices $n-1$ and $n$.
    \[
    \dynkin[labels={1,2,n-2,n-1,n},label directions={,,,,,},scale=2.3]C{}
    \]
    Therefore, if $n \not\in I$, then we remove the unique double edge in the Dynkin diagram and we end up with a disjoint union of diagrams of type $A_{\lambda_{j}-1}$, justifying the fact that $L_I$ is a product of groups $\GL_{\lambda_{j}}$ in (\ref{eq:levi-sp-not-n}). Notice that if $n-1 \not\in I$, we still get a union of Dynkin diagrams of type $A_{\lambda_{j}-1}$ and $A_1$, and this agrees with the fact that in (\ref{eq:levi-sp-n}) we get $L_I = \prod_{j} \GL_{\lambda_j} \times \Sp_{2}$ and $\Sp_2 = \SL_2$.
\end{rem}

Regarding the action of the Weyl group $W = \ZZ_2^n \rtimes S_n$ on $2^\Delta$, notice that if $n \not\in I$, say $I = \{m_1, \ldots, m_t\}$ with $m_t<n$, then the associated roots are $ \{\beta_{m_1(m_1+1), +-},\ldots, \beta_{m_t(m_t+1), +-}\}$; and otherwise if $I = \{m_1, \ldots, m_t, n\}$, then the associated roots are $ \{\beta_{m_1(m_1+1), +-},\ldots, \beta_{m_t(m_t+1), +-}, \alpha_{n,+}\}$. This means that we can decompose $2^\Delta = \Omega_{n} \sqcup \overline{\Omega}_n$, where $\Omega_n$ are the subsets of $\Delta$ containing $n$ and $\overline{\Omega}_n$ those subsets not containing $n$. Observe that both $\Omega_n$ and $\overline{\Omega}_n$ have $2^{n-1}$ elements and the equivalence relation $\sim_W$ can only identify elements in the same stratum.

Hence, the equivalence classes $\overline{\Omega}_n / \sim_W$ are given by unordered partitions $[k] = [1^{k_{1}}\cdots j^{k_{j}}\cdots n^{k_{n}}] \in \mathcal{P}_n$ of $n$, which determine a subset $I_{[k]} \in \overline{\Omega}_n$ giving rise to the Levi subgroup
$$
    L_{[k]} = L_{I_{[k]}} = \prod_{j=1}^n \GL_{j}^{k_j}.
$$
Notice that in this case we have Weyl group $W_{[k]} = \ZZ_2^{|[k]|} \rtimes S_{[k]}$, where we recall that $|[k]|$ is the length of the partition $[k]$, and we have a splitting $N_{[k]} = L_{[k]} \rtimes W_{[k]}$. Therefore, the strata of the character variety corresponding to these Levi subgroups are 
$$
    \mathcal{X}_{L_{[k]}}^\star \sslash (\ZZ_2^{|[k]|} \rtimes S_{[k]}), 
$$
where $S_{[k]}$ acts by permuting blocks of the same dimension and $\ZZ_2^{|[k]|}$ acts by permuting pairs of blocks $A_j \leftrightarrow (A_j^t)^{-1}$ corresponding to an orthogonal pair of subspaces.

Analogously, the equivalence classes $\Omega_n / \sim_W$ are also parametrized by unordered partitions $[k] = [1^{k_{1}}\cdots j^{k_{j}}\cdots m^{k_{m}}] \in \mathcal{P}_m$ of $m$ for some $m < n$, which determine a subset $I_{n, [k]} = I_{[k]} \cup \{n\} \in \Omega_n$ giving rise to the Levi subgroup
$$
    L_{n,[k]} = L_{I_{n,[k]}} = \prod_{j=1}^m \GL_{j}^{k_j} \times \Sp_{2(n-m)}.
$$
Notice that $m = 0$ is acceptable here and, with the convention that $\mathcal{P}_0 = \{\emptyset\}$, it corresponds to the Levi subgroup $L_{n,\emptyset} = \Sp_{2n}$.

Again, the Weyl group $W_{n,[k]} = \ZZ_2^{|[k]|} \rtimes S_{[k]}$ and we have a splitting $N_{n,[k]} = L_{n,[k]} \rtimes W_{n,[k]}$. Therefore, the strata of the character variety corresponding to these Levi subgroups are 
$$
    \mathcal{X}_{L_{n,[k]}}^\star \sslash (\ZZ_2^{|[k]|} \rtimes S_{[k]}).
$$

Therefore, by the results above, we get a decomposition into simpler pieces
\begin{equation}\label{eq:decomposition-Sp}
    \left[ \mathcal{X}_{\Sp_{2n}}(\Gamma) \right] = \sum_{[k]\in\mathcal{P}_{n}} \left[\mathcal{X}_{L_{[k]}}^\star(\Gamma) \sslash (\ZZ_2^{|[k]|} \rtimes S_{[k]}) \right] + \sum_{m=0}^{n-1} \sum_{[k] \in \mathcal{P}_m} \left[\mathcal{X}_{L_{n,[k]}}^\star(\Gamma) \sslash (\ZZ_2^{|[k]|} \rtimes S_{[k]}) \right].
\end{equation}
\subsection{Stratification for $G = \SO_{2n+1}$ (Dynkin diagram $B_n$)}\label{sec:stratification-so-odd}

Recall that this case is Langlands dual to $\Sp_{2n}$. Indeed, with the notation of Section \ref{sec:lie-so-odd} and in agreement with the previous case, the simple roots are $\Delta=\{\beta_{12,+ -}, \beta_{23,+-}, \beta_{34,+-},\ldots, \beta_{(n-1) n, +-}, \alpha_{n, +}\}$, the last one being the unique short root. Again, we index them as $\Delta = \{1,\ldots, n\}$ with $j$ corresponding to $\beta_{j(j+1),+-}$ for $1 \leq j \leq n-1$ and $n$ corresponding to $\alpha_{n,+}$.

Recall that we are considering as quadratic form
$$
    Q(a) = a_1a_{n+2} + a_2a_{n+3} + \cdots + a_{n}a_{2n+1} + a_{n+1}^2,
$$
where $a = (a_1, \ldots, a_{2n+1})$. For convenience, let us relabel the standard basis $e_1, \ldots, e_{2n+1}$ of $\CC^{2n+1}$ as
$$
    x_1 = e_1, \ldots, x_{n} = e_{n}, \quad y_1 = e_{n+2}, \ldots, y_{n} = e_{2n+1}, \quad z = e_{n+1}.
$$

Given a subset $I \subseteq \Delta$, let $\Delta \setminus I = \{i_1, \ldots, i_s\}$ with $i_1 < i_2 < \cdots < i_s$. In this situation, the standard parabolic subgroup $P_I$ associated to $I$ is the subgroup of orthogonal matrices stabilizing the flag
$$
    0 \subsetneq V_{1}  \subsetneq \cdots \subsetneq V_{j} \subsetneq \cdots \subsetneq V_{s} \subsetneq \CC^{2n+1}.
$$
where $V_{j} = \langle x_{1}, \ldots, x_{i_j}\rangle$ for $1 \leq j \leq s$ are isotropic subspaces. But, again, the elements of $P_I$ also preserve the orthogonal complement $V_j^\perp = \{w \in \CC^{2n+1}\,\mid Q(w, v) = 0,\forall v\in V_j\}$, so $P_I$ is also the stabilizer of the flag
\begin{equation}\label{eq:flag-orthogonal-odd}
    0 \subsetneq V_{1} \subsetneq \cdots \subsetneq V_{j} \subsetneq \cdots \subsetneq V_{s} \subsetneq V_{s}^\perp \subsetneq V_{{s-1}}^\perp \subsetneq \cdots \subsetneq V_{1}^\perp \subsetneq \CC^{2n+1}.
\end{equation}
Notice that, explicitly, we have $V_j^\perp = \langle x_1, \ldots, x_n, y_{i_j+1}, \ldots, y_n, z\rangle$.

The Levi subgroup $L_I$ associated to $I$ is the collection of orthogonal maps preserving the graded pieces of the flag (\ref{eq:flag-orthogonal-odd}). Again, we have two situations.
\begin{itemize}
    \item If $n \notin I$, then $\Delta \setminus I = \{i_1, \ldots, i_s = n\}$ with $i_1 < \cdots < i_s = n$. This means that the flag (\ref{eq:flag-orthogonal-odd}) contains the subspace $V_{s} = \langle x_1, \ldots, x_n\rangle$ whose orthogonal is $V_{s}^{\perp} = \langle x_1, \ldots, x_n, z\rangle$. Now, if we set $i_0=0$ and $V_0 = 0$, we can consider the splitting
    $$
        W_j = \langle x_{i_{j-1}+1}, \ldots, x_{i_j}\rangle=V_j / V_{j-1}, \qquad W_j' = \langle y_{i_{j-1}+1}, \ldots, y_{i_j}\rangle=V_{j-1}^{\perp}/V_{j}^{\perp}, \qquad W^\ast = \langle z \rangle = V_{s}^\perp / V_{s},
    $$
    for $j=1,\ldots, s$. In this manner, since both $W_j$ and $W_j'$ are isotropic spaces, the quadratic form vanishes on them, so the only constraint is that if $A$ acts on $W_j$ as a linear endomorphism $A_j \in \GL_{\lambda_j}$, where $\lambda_j = i_{j} -i_{j-1}$, then $A$ acts on $W_j'$ as $(A_j^t)^{-1}$. Moreover, since the orthogonal map is orientation preserving, it must act trivially on $W^\ast$. Hence, we get
    $$
        L_I \cong \prod_{j=1}^s \GL_{\lambda_j},
    $$
    where the isomorphism is given by $\prod_{j=1}^s \GL_{\lambda_j} \ni (A_1, \ldots, A_s) \mapsto (A_1, \ldots, A_s, (A_s^t)^{-1}, \ldots, (A_1^t)^{-1}) \in L_I$.

    \item If $n \in I$, then $\Delta \setminus I = \{i_1, \ldots, i_s\}$ with $i_s < n$. This means that the flag (\ref{eq:flag-orthogonal-odd}) does not contain the subspace $\langle x_1, \ldots, x_n\rangle$. The splitting of (\ref{eq:flag-orthogonal-odd}) is now given again by
    $$
        W_j = \langle x_{i_{j-1}+1}, \ldots, x_{i_j}\rangle, \qquad W_j' = \langle y_{i_{j-1}+1}, \ldots, y_{i_j}\rangle, \qquad    W^\ast = \langle x_{i_s+1}, \ldots, x_n, z,  y_{i_s+1}, \ldots, y_n\rangle,
    $$
    for $1 \leq j \leq s$.
    Then, $L_I$ is the subspace of orthogonal maps preserving $W_j$ and $W_j'$ for all $1 \leq j \leq s$, as well as $W^\ast$. Now, the restriction of the quadratic form to $W^\ast$ gives again an euclidean space of dimension $2(n-i_s)+1$. Hence, we get
    $$
        L_I \cong \prod_{j=1}^s \GL_{\lambda_j} \times \SO_{2(n-i_s)+1},
    $$
    where $\lambda_j = i_{j} -i_{j-1}$.
\end{itemize}

\begin{rem}
    Again, the special cases where $\{n-1, n\} \subseteq I$ can be seen from the fact that the double edge survives in the Dynkin diagram $I$ of $L_I$.
    \[
    \dynkin[labels={1,2,n-2,n-1,n},label directions={,,,,,},scale=2.3]B{}
    \]
\end{rem}

Regarding the Weyl group, as in Section \ref{sec:lie-so-odd}, we have that $W = \ZZ_2^n \rtimes S_n$. Again, we can decompose subsets of simple roots as $2^\Delta = \Omega_{n} \sqcup \overline{\Omega}_n$, where $\Omega_n$ are the subsets of $\Delta$ containing $n$ and $\overline{\Omega}_n$ those subsets not containing $n$. 

As in the previous case, the equivalence classes of $\overline{\Omega}_n / \sim_W$ are given by unordered partitions $[k] = [1^{k_{1}}\cdots j^{k_{j}}\cdots n^{k_{n}}] \in \mathcal{P}_n$ of $n$, which determine a subset $I_{[k]} \in \overline{\Omega}_n$ giving rise to the Levi subgroup
$$
    L_{[k]} = L_{I_{[k]}} = \prod_{j=1}^n \GL_{j}^{k_j}.
$$
The strata of the character variety corresponding to these Levi subgroups are thus
$$
    \mathcal{X}_{L_{[k]}}^\star \sslash (\ZZ_2^{|[k]|} \rtimes S_{[k]}), 
$$
where $S_{[k]}$ acts by permuting blocks of the same dimension and $\ZZ_2^{|[k]|}$ acts by permuting pairs of blocks $A_j \leftrightarrow (A_j^t)^{-1}$ corresponding to an orthogonal pair of subspaces.

Again, equivalence classes $\Omega_n / \sim_W$ are parametrized by unordered partitions $[k] = [1^{k_{1}}\cdots j^{k_{j}}\cdots m^{k_{m}}] \in \mathcal{P}_m$ of $m$ for some $0 \leq m < n$, which determine a subset $I_{n, [k]} = I_{[k]} \cup \{n\} \in \Omega_n$ giving rise to the Levi subgroup
$$
    L_{n,[k]} = L_{I_{n,[k]}} = \prod_{j=1}^m \GL_{j}^{k_j} \times \SO_{2(n-m)+1}.
$$
The strata of the character variety corresponding to these Levi subgroups are 
$$
    \mathcal{X}_{L_{n,[k]}}^\star \sslash (\ZZ_2^{|[k]|} \rtimes S_{[k]}).
$$
Therefore, by the results above, we get a decomposition of the character variety into simpler pieces
\begin{equation}\label{eq:decomposition-SO-odd}
    \left[ \mathcal{X}_{\SO_{2n+1}}(\Gamma) \right] = \sum_{[k]\in\mathcal{P}_{n}} \left[\mathcal{X}_{L_{[k]}}^\star(\Gamma) \sslash (\ZZ_2^{|[k]|} \rtimes S_{[k]}) \right] + \sum_{m=0}^{n-1} \sum_{[k] \in \mathcal{P}_m} \left[\mathcal{X}_{L_{n,[k]}}^\star(\Gamma) \sslash (\ZZ_2^{|[k]|} \rtimes S_{[k]}) \right].
\end{equation}

\begin{rem}
    Decomposition (\ref{eq:decomposition-SO-odd}) for $\SO_{2n+1}$ is completely analogous to decomposition (\ref{eq:decomposition-Sp}) for $\Sp_{2n}$, as predicted by the fact that these groups are Langlands duals. In particular, the $|\mathcal{P}_n|$ terms corresponding to the strata $L_{[k]}$ are literally equal in both cases. For the other terms, if we denote $\GL_{[k]} = \prod_{j=1}^m \GL_{j}^{k_j}$, then a classical homological mirror symmetry statement reduces to showing that the strata
    $$
    \mathcal{X}_{\GL_{[k]} \times \Sp_{2(n-m)}}^\star(\Gamma) \sslash (\ZZ_2^{|[k]|} \rtimes S_{[k]})\quad \textrm{and}\quad \mathcal{X}_{\GL_{[k]}\times \SO_{2(n-m)+1}}^\star(\Gamma) \sslash (\ZZ_2^{|[k]|} \rtimes S_{[k]}).
    $$
    have related $E$-polynomials.
\end{rem}

\subsection{Stratification for $G = \SO_{2n}$ (Dynkin diagram $D_n$)} Let $n \geq 4$. With the notation of Section \ref{sec:lie-so-even}, the simple roots are $\Delta=\{\beta_{12,+ -}, \beta_{23,+-}, \beta_{34,+-},\ldots, \beta_{(n-1) n, +-}, \beta_{(n-1) n, ++}\}$, the last two roots being the branching in the Dynkin diagram. As always, we relabel them as $\Delta = \{1, \ldots, n\}$ with $j$ corresponding to $\beta_{j(j+1),+-}$ for $j \leq n-1$ and $n$ corresponding to $\beta_{(n-1) n, ++}$.

In this case, recall that we are considering as quadratic form
$$
    Q(a) = a_1a_{n+1} + a_2a_{n+2} + \cdots + a_{n}a_{2n},
$$
where $a = (a_1, \ldots, a_{2n})$. We relabel the standard basis $e_1, \ldots, e_{2n}$ as
$$
    x_1 = e_1, \ldots, x_{n} = e_{n}, \quad y_1 = e_{n+1}, \ldots, y_{n} = e_{2n}.
$$

Recall that the main feature in this case is that there no longer exists a bijection between parabolic subgroups of $\SO_{2n}$ and isotropic flags, so a modification of the description in Section \ref{sec:stratification-so-odd} must be taken into account. Following \cite{conrad2020reductive}, the main difficulty is the following. Suppose that we have a flag
$$
    F: 0 \subsetneq V_{1}  \subsetneq \cdots \subsetneq V_{j} \subsetneq \cdots \subsetneq V_{s} \subsetneq \CC^{2n},
$$
where each $V_j$ is isotropic and $\dim V_{s} = n-1$. Now, notice that $V_s^\perp/V_s$ is an euclidean plane and thus contains exactly two (different) isotropic lines. Hence, $V_s$ can be completed into an $n$-dimensional isotropic space in two different ways, let us call them $V'$ and $V''$. Now, the trouble appears from the fact that the stabilizer of the flag $F$ coincides with the stabilizer of the flags
$$
    F': 0 \subsetneq V_{1}  \subsetneq \cdots \subsetneq V_{j} \subsetneq \cdots \subsetneq V_{s} \subseteq V' \subsetneq \CC^{2n},
$$
$$
    F'': 0 \subsetneq V_{1}  \subsetneq \cdots \subsetneq V_{j} \subsetneq \cdots \subsetneq V_{s} \subseteq V'' \subsetneq \CC^{2n}.
$$
On the other hand, if we remove the $V_s$ space of dimension $n-1$, then the two resulting flags
$$
    \tilde{F}': 0 \subsetneq V_{1}  \subsetneq \cdots \subsetneq V_{j} \subsetneq \cdots \subsetneq V_{s-1} \subseteq V' \subsetneq \CC^{2n},
$$
$$
    \tilde{F}'': 0 \subsetneq V_{1}  \subsetneq \cdots \subsetneq V_{j} \subsetneq \cdots \subsetneq V_{s-1} \subseteq V'' \subsetneq \CC^{2n},
$$
induce different parabolic subgroups that are not conjugate.

There are several ways of circumventing this issue. Consider a subset $I \subseteq \Delta$ and let $\Delta \setminus I = \{i_1, \ldots, i_s\}$. Then we have the following.

\begin{enumerate}
    \item If $\{n-1, n\} \subseteq I$, then $i_s \leq n-2$. In this case, $P_I$ is just the parabolic subgroup which is the stabilizer of the flag
    $$
        0 \subsetneq V_{1} \subsetneq \cdots \subsetneq V_{j} \subsetneq \cdots \subsetneq V_{s-1} \subsetneq V_{s} \subsetneq V_{s}^\perp \subsetneq V_{{s-1}}^\perp \subsetneq \cdots \subsetneq V_{1}^\perp \subsetneq \CC^{2n}.
    $$
    Here, $V_j = \langle x_1, \ldots, x_{i_j}\rangle$ and thus its orthogonal complement is $V_j^\perp = \langle x_1, \ldots, x_n, y_{i_j+1}, \ldots, y_n\rangle$. We can easily compute the associated Levi subgroup from this, since
    $$
        W_j = \langle x_{i_{j-1}+1}, \ldots, x_{i_j}\rangle, \qquad W_j' = \langle y_{i_{j-1}+1}, \ldots, y_{i_j}\rangle, \qquad  W^\ast = \langle x_{i_s+1}, \ldots, x_n, y_{i_s+1}, \ldots, y_n\rangle
    $$
    for $1 \leq j \leq s$ provide a splitting of the flag. Therefore, we get that the associated Levi subgroup is
    $$
        L_I \cong \prod_{j=1}^s \GL_{\lambda_j} \times \SO_{2(n-i_s)},
    $$
    where $\lambda_j = i_{j} -i_{j-1}$. Notice that $i_s \leq n-2$, so the last factor is a genuine $\SO_{2m}$ with $m \geq 2$. Additionally, if we allow $s = 0$ (and $i_s = 0$), we also cover the case $I = \Delta$, in such a way that the associated Levi subgroup is the stabilizer of the trivial flag and thus $L_{\Delta} = \SO_{2n}$.
    \item If $n \not\in I$ and $n-1 \in I$, then $i_{s-1} < n-1$ and $i_s = n$. Again, $P_I$ is the stabilizer of the flag
    $$
        0 \subsetneq V_{1} \subsetneq \cdots \subsetneq V_{j} \subsetneq \cdots \subsetneq V_{s-1} \subsetneq V_{s} = V_{s}^\perp \subsetneq V_{{s-1}}^\perp \subsetneq \cdots \subsetneq V_{1}^\perp \subsetneq \CC^{2n},
    $$
    where $V_j = \langle x_1, \ldots, x_{i_j}\rangle$. However, since $V_s = \langle x_1, \ldots, x_n\rangle = V_s^\perp$, we have that a splitting is just
    $$
        W_j = \langle x_{i_{j-1}+1}, \ldots, x_{i_j}\rangle, \qquad W_j' = \langle y_{i_{j-1}+1}, \ldots, y_{i_j}\rangle,
    $$
    for $1 \leq j \leq s$. Thus, the associated Levi subgroup is
    $$
        L_I \cong \prod_{j=1}^s \GL_{\lambda_j},
    $$
    where $\lambda_j = i_{j} -i_{j-1}$.
    \item If $n \in I$ and $n-1 \not\in I$, then $i_{s-1} < n-1$ and $i_s = n-1$. This is the case that must be modified to get a different parabolic subgroup. Now, we consider $V_j = \langle x_1, \ldots, x_{i_j}\rangle$ for $j < s$ and set $V_s' = \langle x_1, \ldots, x_{n-1}, y_n \rangle$, which is a self-orthogonal subspace. We get that $P_I$ is the stabilizer of the flag
    $$
        0 \subsetneq V_{1} \subsetneq \cdots \subsetneq V_{j} \subsetneq \cdots \subsetneq V_{s-1} \subsetneq V_s' = V_{s}'^\perp \subsetneq V_{{s-1}}^\perp \subsetneq \cdots \subsetneq V_{1}^\perp \subsetneq \CC^{2n}.
    $$
    Now, the splitting of the flag is given by
    $$
        W_j = \langle x_{i_{j-1}+1}, \ldots, x_{i_j}\rangle, \qquad W_j' = \langle y_{i_{j-1}+1}, \ldots, y_{i_j}\rangle,
    $$
    for $1 \leq j \leq s-1$, as well as the spaces
    $$
        W^\ast = \langle x_{i_{s-1}+1}, \ldots, x_{n-1}, y_{n} \rangle, \qquad W^{\ast\ast} = \langle x_{n}, y_{i_{s-1}+1}, \ldots, y_{n-1} \rangle.
    $$
    The situation is thus the same as in (2), so the associated Levi subgroup is
    $$
        L_I \cong \prod_{j=1}^{s} \GL_{\lambda_j},
    $$
    where $\lambda_j = i_{j} -i_{j-1}$ for $1 \leq j \leq s-1$ and $\lambda_s = n-i_{s-1}$.
    \item If $n \not\in I$ and $n-1 \not\in I$, then $i_{s-1} = n-1$ and $i_s = n$. Then, $P_I$ is the stabilizer of the flag
    $$
        0 \subsetneq V_{1} \subsetneq \cdots \subsetneq V_{j} \subsetneq \cdots \subsetneq V_{s-1} \subsetneq V_{s} = V_{s}^\perp \subsetneq V_{{s-1}}^\perp \subsetneq \cdots \subsetneq V_{1}^\perp \subsetneq \CC^{2n},
    $$
    where $V_j = \langle x_1, \ldots, x_{i_j}\rangle$. Again, since $V_s = \langle x_1, \ldots, x_n\rangle = V_s^\perp$, we have that a splitting is just
    $$
        W_j = \langle x_{i_{j-1}+1}, \ldots, x_{i_j}\rangle, \qquad W_j' = \langle y_{i_{j-1}+1}, \ldots, y_{i_j}\rangle,
    $$
    for $1 \leq j \leq s$. Thus, the associated Levi subgroup is
    $$
        L_I \cong \prod_{j=1}^s \GL_{\lambda_j},
    $$
    where $\lambda_j = i_{j} -i_{j-1}$.
\end{enumerate}

\begin{rem}
    The division in the cases above can be easily understood from the shape of the Dynkin diagram $D_n$, which is a chain of the roots $1, 2, \ldots, n-2$ branching at this last vertex so that $n-2$ is connected to both the roots $n-1$ and $n$. 
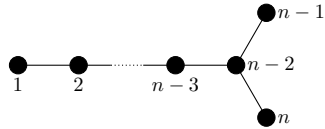
    \[
    \dynkin[labels={1,2,n-3,n-2,n-1,n},label directions={,,,right,,},scale=2.3]D{}
    \]
    In this manner, if we remove $n$ or $n-1$ from $I \subseteq \Delta$, the resulting Dynkin diagram is a disjoint union of diagrams of type $A_{\lambda_j-1}$, agreeing with the shape of the Levi subgroup. The only genuinely new Levi subgroups arise when $\{n-1,n\} \subseteq I$, corresponding to case (1), in which a new factor $\SO_{2m}$ appears with $m \geq 2$. Furthermore, observe that $\SO(4)$ (resp.\ $\SO(6)$), which correspond to $n-2 \not\in I$ (resp.\ $n-2 \in I$ and $n-3 \not\in I$), is actually of type $A_1 \times A_1$ (resp.\ of type $A_3$), in agreement with the fact that the resulting Dynkin diagram has no bifurcation.
\end{rem}

With respect to the equivalence relation $\sim_W$ on $2^{\Delta}$, we follow \cite[Section 3]{douglass2013reflection}. Observe that each subset $I \subseteq \Delta$ determines the number
$$
    m := \left\{\begin{array}{ll}\max \{m \geq 2 \,\mid \, \{n-m+1, n-m+2, \ldots, n\} \subseteq I\} & \textrm{if } \{n-1, n\} \subseteq I,\\
    0 & \textrm{otherwise}.
    \end{array}
    \right.
$$
Notice that $m$ determines the number of nodes in the $D_m$-tail of $I$ and, in particular, $m \neq 1$. For $m \geq 2$, the set $I_m = \{1, 2, \ldots, n-m-1\} \cap I$ corresponds to a subset of a Dynkin diagram of type $A_{n-m-1}$ and, thus, defines a partition $[k] = [1^{k_{1}}\cdots j^{k_{j}}\cdots (n-m)^{k_{n-m}}] \in \mathcal{P}_{n-m}$. In the case $m = 0$, $I$ itself corresponds to a collection of diagrams of type $A$ and thus defines a partition $[k] = [1^{k_{1}}\cdots j^{k_{j}}\cdots n^{k_{n}}] \in \mathcal{P}_n$. Since $m$ and $[k]$ determine the shape of the resulting subdiagram, they are invariants under the equivalence relation $\sim_W$ induced by the Weyl group $W$.

Reciprocally, a pair $(m, [k])$ with $m \geq 2$ and $[k] \in \mathcal{P}_{n-m}$ determines a subset $I_{m, [k]} := I_{[k]} \cup \{n-m+1, \ldots, n-1,n\} \subseteq \Delta$, and thus the Levi subgroup of type (1)
$$
    L_{m,[k]} := L_{I_{m,[k]}} = \prod_{j=1}^{n-m} \GL_{j}^{k_j} \times \SO_{2m}.
$$
In this case, this subset $I_{m, [k]}$ is a canonical representative in the equivalence class under $\sim_W$. Furthermore, the Weyl group in this case is $W_{m,k} = \ZZ_2^{|[k]|} \rtimes S_{[k]}$, where $|[k]|$ is the length of the partition $[k]$ and $S_{[k]} = \prod_j S_{k_j}$.

However, in the case $m = 0$ the situation is slightly more involved. In this case, we have a pair $(0, [k])$, with $[k] = [1^{k_{1}}\cdots j^{k_{j}}\cdots n^{k_{n}}] \in \mathcal{P}_n$. Considering the subset $I_{[k]}$ associated to $[k]$, we have that $I_{[k]}$ does not contain $n-1$ and $n$ simultaneously (otherwise, $m \neq 0$). Then, we have two options:
\begin{itemize}
    \item If $\{n-1, n\} \cap I_{[k]} = \emptyset$, then it defines the Levi subgroup of type (4)
$$
    L_{0,[k]} := L_{I_{[k]}} = \prod_{j=1}^{n} \GL_{j}^{k_j}.
$$
This is a canonical representative for the relation $\sim_W$.
    \item $n \in I_{[k]}$ and $n-1 \not\in I_{[k]}$, or reciprocally $n \not\in I_{[k]}$ and $n-1 \in I_{[k]}$. Observe that, in type $D_n$, the vertices $n$ and $n-1$ play symmetric roles. For this reason, it makes sense to consider instead the sets
    $$
        I_{[k]}^+ := I_{[k]} \setminus  \{n,n-1\} \cup \{n-1\}, \qquad  I_{[k]}^- := I_{[k]} \setminus \{n,n-1\} \cup \{n\}.
    $$
    Both sets correspond to a Levi subgroup of the form (2) or (3), respectively
    $$
        L_{0,[k]}^+ := L_{I_{[k]}^+}, \qquad L_{0,[k]}^- := L_{I_{[k]}^-}.
    $$
    Notice that in both cases, these are groups of the form $\prod_j \GL_{\lambda_j}$. Now, whether these two subsets are equivalent or not depends on the parity of the partition $[k]$.
    \begin{enumerate}
        \item If $[k]$ is totally even (i.e.\ $k_j = 0$ for all odd $j$, so only even numbers appear in the partition), then $I_{[k]}^+ \not\sim_W I_{[k]}^-$. The reason is that, in this case, the Weyl group cannot send the subsystem $\Phi_{I_{[k]}^+}$ containing the root $n-1 = \beta _{(n-1)n,+-}$ into the subsystem $\Phi_{I_{[k]}^-}$ containing the root $n = \beta _{(n-1)n,++}$, since this would require a signed permutation of odd sign parity, not allowed in $W = H_{n} \rtimes S_n$.
        \item If $[k]$ contains some odd part (i.e.\ $k_j \neq 0$ for some odd $j$), then $I_{[k]}^+ \sim_W I_{[k]}^-$. Indeed, now this odd factor $\GL_{j}$ can be used to invert the orientation under the action of $W$, at the same time we permute the root $n-1 = \beta _{(n-1)n,+-}$ with the root $n = \beta _{(n-1)n,++}$. In this case, we then denote $L_{0,[k]} := L_{I_{[k]}}$.
    \end{enumerate}
    We emphasize that, if $n$ is odd, then (2) always holds and thus case (1) never happens. 
\end{itemize}
In all these cases, the associated Weyl group is $W_{0, [k]} = H_{[k]} \rtimes S_{[k]}$, where $H_{[k]}$ is the subgroup of $\ZZ_2^{|[k]|}$ given by
$$
    H_{[k]} = \left\{(\epsilon_{j,\ell}) \in \prod_{j = 1}^n \ZZ_2^{k_j} = \ZZ_2^{|[k]|} \,\left|\, \prod_{j, \ell} \epsilon_{j,\ell}^j = 1 \right.\right\}.
$$

Therefore, by the previous discussion, we get the decomposition 
\begin{align}\label{eq:decomposition-SO-even}
    \left[ \mathcal{X}_{\SO_{2n}}(\Gamma) \right] &= \sum_{m = 2}^{n}\sum_{[k]\in\mathcal{P}_{n-m}} \left[\mathcal{X}_{L_{m,[k]}}^\star(\Gamma) \sslash (\ZZ_2^{|[k]|} \rtimes S_{[k]}) \right] + \sum_{[k]\in\mathcal{P}_{n}^{\textup{odd}}} \left[\mathcal{X}_{L_{0,[k]}}^\star(\Gamma) \sslash (H_{[k]} \rtimes S_{[k]}) \right] \\
    &+ \sum_{[k]\in\mathcal{P}_{n}^{\textup{even}}} \left(\left[\mathcal{X}_{L_{0,[k]}^+}^\star(\Gamma) \sslash (H_{[k]} \rtimes S_{[k]}) \right]+\left[\mathcal{X}_{L_{0,[k]}^-}^\star(\Gamma) \sslash (H_{[k]} \rtimes S_{[k]}) \right]\right).\nonumber
\end{align}
Here, $\mathcal{P}_{n}^{\textup{even}}$ denotes the set of totally even partitions of $n$, whereas $\mathcal{P}_{n}^{\textup{odd}}$ is the set of partitions of $n$ with some odd part. Again, recall that if $n$ is odd, then $\mathcal{P}_{n} = \mathcal{P}_{n}^{\textup{odd}}$.

\subsection{The role of the group $\Gamma$}

As discussed in the previous sections, Theorem \ref{thm:main-result} and Corollary \ref{cor:main-result-good-situation} allow us to stratify the character variety $\mathcal{X}_G(\Gamma)$ depending on the root data of $G$. This way, one could think that such decompositions are independent of the group $\Gamma$ to be represented.

In this section, we will show that the group $\Gamma$ does affect the stratification of Theorem \ref{thm:main-result}. Particularly, we will show how different choices of $\Gamma$ may drastically change the geometry of each of the strata involved, including the possibility that some strata may become empty. We will exemplify this behaviour for $G = \SL_{2}$ in three cases: the free group $\Free{2}$ on $2$ generators, the free abelian group $\ZZ^2$ of rank $2$, and the trefoil group $\Gamma_{2,3}$.

Recall that a Borel subgroup for $\SL_2$ is the set of upper triangular matrices, and its associated maximal torus is the set of diagonal matrices
$$
    B = \left\{\left.\begin{pmatrix}
        a & b \\
        0 & a^{-1}
    \end{pmatrix}\,\right|\, a \neq 0\right\} \subseteq \SL_2, \qquad 
    T = \left\{D_a := \left.\begin{pmatrix}
        a & 0 \\
        0 & a^{-1}
    \end{pmatrix}\,\right|\, a \neq 0\right\} \cong \mathbb{G}_m \subseteq \SL_2.
$$
In particular, its character lattice is $X^*(T) \cong \ZZ$, generated by $\chi: T \to \mathbb{G}_m, \; D_a\mapsto \chi(D_a) = a$. It contains a single simple root $\alpha := \alpha_{12}$ given by $\alpha(D_a) = a\cdot (a^{-1})^{-1}=a^2$, so $\Delta = \{\alpha\}$,  and the root lattice is $2\ZZ \subseteq \ZZ$. This agrees with the fact that the Dynkin diagram of $\SL_2$ is $A_1$.

The possible subsets $I \subseteq \Delta$ are thus $I = \emptyset$ and $I = \Delta$. As explained in Remark \ref{ex:subset-extreme}, the former case corresponds to parabolic and Levi subgroups given by the Borel subgroup and maximal torus, respectively, 
$$
    P_{\emptyset} = B, \qquad T_\emptyset = L_\emptyset =T.
$$
Additionally, the associated Weyl group is the standard Weyl group $W_{\emptyset} = W = S_2$, acting on $T$ by permutation of the diagonal entries.

On the other hand, the latter case $I = \Delta$ corresponds to
$$
    P_{\Delta} = \SL_2, \qquad T_\Delta = \left\{I\right\}, \qquad L_\Delta = \SL_2,
$$
as $T_\Delta$ is the identity component of the center of $\SL_2$. In this case, the associated Weyl group is $W_{\Delta} = N_{\SL_2}(T_{\Delta})/Z_{\SL_2}(T_{\Delta}) = \SL_2/\SL_2 = 1$.

In the case $I = \Delta$, the sequence (\ref{eq:split-weyl}) trivially splits; whereas for $I = \emptyset$ the discussion at the end of Section \ref{sec:parabolic-stratification} still applies. Hence, using Corollary \ref{cor:main-result-good-situation}, we have
\begin{equation}\label{eq:decom-sl2}
    \left[\mathcal{X}_{\SL_2}(\Gamma)\right] = \underbrace{\left[\mathcal{R}_{T}(\Gamma) \sslash S_2\right]}_{I = \emptyset} + \underbrace{\left[\mathcal{X}_{\SL_2}^\star(\Gamma)\right]}_{I = \Delta}.
\end{equation}
Notice that these strata perfectly match the general description of Section \ref{sec:SL-PGL}, since they are in correspondence with the partitions $\mathcal{P}_2 = \{[1^2],[2^1]\}$ of $2$ as
$$
    \mathcal{R}_{T}(\Gamma) \sslash S_2 = \mathcal{X}^{[1^2]}_{\SL_2}(\Gamma) \; \textrm{ and }\;\mathcal{X}_{\SL_2}^\star(\Gamma) = \mathcal{X}^{[2^1]}_{\SL_2}(\Gamma).
$$

Additionally, we can go a step further in (\ref{eq:decom-sl2}) and identify each stratum more explicitly as follows:
\begin{itemize}
    \item For $\mathcal{R}_{T}(\Gamma) \sslash S_2$, observe that $T \cong \mathbb{G}_m$ under the isomorphism $D_a \mapsto a$. In this manner $\mathcal{R}_{T}(\Gamma)$ is isomorphic to $\mathcal{R}_{\mathbb{G}_m}(\Gamma)$ and the action of $S_2$ becomes $\rho \mapsto \rho^{-1}$ for $\rho \in \mathcal{R}_{\mathbb{G}_m}(\Gamma)$, where $\rho^{-1}: \Gamma \to \mathbb{G}_m$ is the representation given by $\rho^{-1}(\gamma) = (\rho(\gamma))^{-1}$ for $\gamma \in \Gamma$.
    \item For $\mathcal{X}_{\SL_2}^\star(\Gamma) = \mathcal{R}_{\SL_2}^\star(\Gamma) \sslash \SL_2$, let us choose a set of generators $\gamma_1 ,\ldots, \gamma_s$ for $\Gamma$. Then, a representation $\rho: \Gamma \to \SL_2$ is represented by a tuple $(A_1, \ldots, A_s) := (\rho(\gamma_1), \ldots,\rho(\gamma_s)) \in \SL_2^s$ satisfying the $\Gamma$ relations, so that $\mathcal{R}_{\SL_2}^\star(\Gamma) \subseteq \SL_2^s$. With this description, a representation $(A_1, \ldots, A_s) \in \mathcal{R}_{\SL_2}(\Gamma)$ is irreducible if and only if there exists no proper parabolic subgroup $P \subsetneq \SL_2$ with $A_1, \ldots, A_s \in P$. But all proper parabolic subgroups of $\SL_2$ are conjugate to the standard Borel subgroup $B$ of upper triangular matrices, so the irreducibility is equivalent to $A_1, \ldots, A_s$ having no common eigenvector.
\end{itemize}

Despite these general considerations, the geometry of $\mathcal{R}_{T}(\Gamma)$ and $\mathcal{R}_{\SL_2}^\star(\Gamma)$ strongly depends on the group $\Gamma$, as the following examples show.

\subsubsection{The free group $\Free{2}$} Let us take $\Gamma = \Free{2}$ to be the free group generated by two elements, let us call them $x$ and $y$. This case has been analyzed in depth in the literature, for instance in \cite{florentino2009topology,florentino2021serre, gonzalez2024pseudo,lawton2008minimal,lawton2016polynomial}.

Observe that, in this case, a representation $\rho: \Free{2} \to G$ into a reductive group $G$ is fully characterized by the pair of images $(\rho(x), \rho(y)) \in G^2$, which is unconstrained. In this manner, the corresponding representation and character varieties are given by
$$
    \mathcal{R}_{G}(\Free{2}) = G^2, \qquad \mathcal{X}_{G}(\Free{2}) = G^2 \sslash G,
$$
where $G$ acts on $G^2$ by simultaneous conjugation.

In the particular case that $G = \SL_2$, by (\ref{eq:decom-sl2}) we have
$$
[\SL_2^2 \sslash \SL_2] = [\mathcal{X}_{\SL_2}(\Free{2})] = \left[\mathcal{R}_{\mathbb{G}_m}(\Free{2}) \sslash S_2\right]+ \left[\mathcal{X}_{\SL_2}^\star(\Free{2})\right].
$$
Each stratum can be analyzed as follows:
\begin{itemize}
    \item For $\mathcal{R}_{\mathbb{G}_m}(\Free{2}) = \mathbb{G}_m^2$, the action is $(a_1, a_2) \mapsto (a_1^{-1}, a_2^{-1})$. It is straightforward to check that this quotient is isomorphic to the cubic surface
    $$
        \mathbb{G}_m^2 / S_2 \cong \left\{(x,y,z) \in \mathbb{A}^3\,\mid\, x^2+y^2+z^2-xyz=4\right\}.
    $$
    In particular, its virtual class can be easily computed to be $\mathbb{L}^2+1$, with $\mathbb{L} = [\mathbb{A}^1_k] \in \KVar$ the class of the affine line.
    \item For $\mathcal{X}_{\SL_2}^\star(\Free{2}) = \mathcal{R}_{\SL_2}^\star(\Free{2}) \sslash \SL_2$, the variety $\mathcal{R}_{\SL_2}^\star(\Free{2}) \subseteq \SL_2^2$ is the set of pairs of matrices $(A, B) \in \SL^2_2$ not sharing an eigenvector. In particular, $\mathcal{X}_{\SL_2}^\star(\Free{2}) \neq \emptyset$ and, actually, by \cite[Section 7.2]{gonzalez2024pseudo}, we have that its class is
    $$
    [\mathcal{X}_{\SL_2}^\star(\Free{2})] = \mathbb{L}^3 - \mathbb{L}^2 - 1.
    $$
\end{itemize}

Summarizing, in this case both strata contribute to the virtual class of $\mathcal{X}_{\SL_2}(\Free{2})$ and we have
$$
    [\mathcal{X}_{\SL_2}(\Free{2})] = [\SL_2^2 \sslash \SL_2] = \mathbb{L}^3.
$$
Notice that, in this case, the two strata have different dimensions: the reducible character variety $\mathcal{R}_{\mathbb{G}_m}(\Free{2}) \sslash S_2$ has dimension $2$, whereas the irreducible locus $\mathcal{X}_{\SL_2}^\star(\Free{2})$ has dimension $3$.

\subsubsection{The abelian group $\ZZ^2$} Now, let us consider the case $\Gamma = \ZZ^2$, the free abelian group of rank two. This case has been studied in the literature, for instance, in \cite{florentino2021hodgeabelian,sikora2014character}. 
In this case, observe that the images $A := \rho(1,0)$ and $B := \rho(0,1)$ of a representation $\rho: \ZZ^2 \to G$ are not arbitrary elements, since they must commute. In this way, the corresponding representation and character varieties are given by
$$
    \mathcal{R}_{G}(\ZZ^2) = \{(A, B) \in G^2 \,\mid\, AB=BA\}, \qquad \mathcal{X}_{G}(\ZZ^2) = \mathcal{R}_{G}(\ZZ^2) \sslash G,
$$
where $G$ acts on $\mathcal{R}_{G}(\ZZ^2) \subseteq G^2$ by simultaneous conjugation. Therefore, for $G = \SL_2$, by (\ref{eq:decom-sl2}) we have a decomposition
$$
[\mathcal{X}_{\SL_2}(\ZZ^2)] = \left[\mathcal{R}_{\mathbb{G}_m}(\ZZ^2) \sslash S_2\right]+ \left[\mathcal{X}_{\SL_2}^\star(\ZZ^2)\right].
$$

However, in this case, the stratum $\mathcal{X}_{\SL_2}^\star(\ZZ^2) = \mathcal{R}_{\SL_2}^\star(\ZZ^2) \sslash \SL_2$ is empty. Indeed, an element $(A, B) \in \mathcal{R}_{\SL_2}^\star(\ZZ^2)$ would be a pair of commuting matrices not sharing an eigenvector, which is impossible. On the other hand, $\mathcal{R}_{\mathbb{G}_m}(\ZZ^2) = \mathcal{R}_{\mathbb{G}_m}(\Free{2})$ since $\mathbb{G}_m$ is commutative, and thus its contribution is the same as for $\Free{2}$. Therefore, we have
$$
    [\mathcal{X}_{\SL_2}(\ZZ^2)] = \left[\mathcal{R}_{\mathbb{G}_m}(\ZZ^2) \sslash S_2\right] = \left[\mathcal{R}_{\mathbb{G}_m}(\Free{2}) \sslash S_2\right] = \mathbb{L}^2 + 1.
$$
In particular, this example shows that not all the strata in Theorem \ref{thm:main-result} have to be non-empty.

\subsubsection{The trefoil group $\Gamma_{2,3}$} Let us consider the group $\Gamma = \Gamma_{2,3} = \langle x, y \,\mid\, x^2 = y^3\rangle$. This group is usually known as the trefoil group since it is the fundamental group of the complement of the trefoil knot in $\mathbb{R}^3$. Its character varieties have been studied in \cite{gonzalez2022motive,munoz2009,munoz2016geometry}, among others.

In this case, a representation $\rho: \Gamma_{2,3} \to G$ is given by a pair of elements $(A, B) := (\rho(x), \rho(y)) \in G^2$ with the property that $A^2 = B^3$. Hence, we get
$$
    \mathcal{R}_{G}(\Gamma_{2,3}) = \{(A, B) \in G^2 \,\mid\, A^2=B^3\}, \qquad \mathcal{X}_{G}(\Gamma_{2,3}) = \mathcal{R}_{G}(\Gamma_{2,3}) \sslash G.
$$
As always, for $G = \SL_2$, equation (\ref{eq:decom-sl2}) leads to a decomposition
$$
[\mathcal{X}_{\SL_2}(\Gamma_{2,3})] = \left[\mathcal{R}_{\mathbb{G}_m}(\Gamma_{2,3}) \sslash S_2\right]+ \left[\mathcal{X}_{\SL_2}^\star(\Gamma_{2,3})\right].
$$

Let us discuss each stratum separately:
\begin{itemize}
    \item For $\mathcal{R}_{\mathbb{G}_m}(\Gamma_{2,3})$, this is the set of tuples $(a_1, a_2) \in \mathbb{G}_m^2$ with $a_1^2=a_2^3$, endowed with the action of $S_2$ given by $(a_1, a_2) \mapsto (a_1^{-1}, a_2^{-1})$. This space is isomorphic to $\mathbb{G}_m$ via the isomorphism \linebreak $\mathbb{G}_m \ni t \mapsto (t^3, t^2) \in \mathcal{R}_{\mathbb{G}_m}(\Gamma_{2,3})$. Under this isomorphism, the action becomes $t \mapsto t^{-1}$, so the associated quotient is $\mathcal{R}_{\mathbb{G}_m}(\Gamma_{2,3}) \sslash S_2 = \mathbb{G}_m \sslash S_2 \cong \mathbb{A}^1$. Hence, its virtual class is $[\mathcal{R}_{\mathbb{G}_m}(\Gamma_{2,3}) \sslash S_2] = \mathbb{L}$.
    
    \item For $\mathcal{X}_{\SL_2}^\star(\Gamma_{2,3}) = \mathcal{R}_{\SL_2}^\star(\Gamma_{2,3}) \sslash \SL_2$, the variety $\mathcal{R}_{\SL_2}^\star(\Gamma_{2,3}) \subseteq \SL_2^2$ is the set of pairs of matrices $(A, B) \in \SL^2_2$ not sharing an eigenvector and such that $A^2 = B^3$. In this case, observe that both $A^2$ and $B^3$ are intertwining maps of the representation so, by Schur's lemma, they must be multiples of the identity $I$. Furthermore, as $\det(A^2) = \det(B^3) = 1$, the only options are that $A^2 = B^3 = I$ or $A^2 = B^3 = -I$. However, notice that in the former case, $A = \pm I$, so the representation is automatically reducible. Therefore, 
    $$
        \mathcal{R}_{\SL_2}^\star(\Gamma_{2,3}) = \left\{(A, B) \in \SL_2^2\,\left|\, \begin{matrix}\textrm{$A$ and $B$ do not share an eigenvector,} \\ A^2 = B^3 = -I \end{matrix}\right.\right\},
    $$
    and $[\mathcal{X}_{\SL_2}^\star(\Gamma_{2,3})] = [\mathcal{R}_{\SL_2}^\star(\Gamma_{2,3}) \sslash \SL_2]$. This irreducible character variety has been studied in \cite[Proposition 2.3]{munoz2009}, proving that $[\mathcal{X}_{\SL_2}^\star(\Gamma_{2,3})]= \mathbb{L} - 2$. 
\end{itemize}

In particular, putting everything together, we get
$$
[\mathcal{X}_{\SL_2}(\Gamma_{2,3})] = \left[\mathcal{R}_{\mathbb{G}_m}(\Gamma_{2,3}) \sslash S_2\right]+ \left[\mathcal{X}_{\SL_2}^\star(\Gamma_{2,3})\right] = 2\mathbb{L} - 2.
$$
In particular, observe that, in this case, both strata are non-empty and have the same dimension. A graphical description, extracted from \cite{munoz2009}, of the character variety $\mathcal{X}_{\SL_2}(\Gamma_{2,3})$ is depicted in Figure \ref{fig:X-torus}.

\begin{figure}[h]
    \centering
    
\begin{tikzpicture}[
    thick,
    line cap=round,
    line join=round
]

\draw
    ({3+3*cos(45)},{1.5+1.4*sin(45)})
    arc[
        start angle=45,
        end angle=315,
        x radius=3cm,
        y radius=1.4cm
    ];

\draw (2.25,-0.3) -- (2.25,3.2);

\node[left=2pt] at (0,1.5) {$\mathcal{R}_{\mathbb{G}_m}(\Gamma_{2,3}) \sslash S_2$};

\node[below=7pt] at (2.25,-0.15) {$\mathcal{X}_{\SL_2}^\star(\Gamma_{2,3})$};

\end{tikzpicture}
    \caption{Geometry of the character variety $\mathcal{X}_{\SL_2}(\Gamma_{2,3})$}
    \label{fig:X-torus}
\end{figure}

\bibliographystyle{plain} %
\bibliography{bibliography}
\newpage

\end{document}